\pdfoutput=1
\RequirePackage{ifpdf}
\ifpdf 
\documentclass[pdftex]{sigma}
\else
\documentclass{sigma}
\fi

\usepackage{mathrsfs}
\usepackage{faktor}
\usepackage{xfrac}
\usepackage[all]{xy}

\DeclareFontFamily{U}{mathx}{\hyphenchar\font45}
\DeclareFontShape{U}{mathx}{m}{n}{
  <5> <6> <7> <8> <9> <10>
  <10.95> <12> <14.4> <17.28> <20.744> <24.888>
  mathx10
}{}
\DeclareSymbolFont{mathx}{U}{mathx}{m}{n}
\DeclareMathAccent{\widecheck}{0}{mathx}{"71}

\newcommand{\tsp}[1]{{}^t \! #1} 
\newcommand{\ii}{\mathrm{i}} 
\newcommand{\Span}{\operatorname{Span}} 
\newcommand{\supp}{\operatorname{supp}} 
\newcommand{\im}{\operatorname{im}} 
\newcommand*\diff{\mathop{}\!\mathrm{d}} 
\newcommand{\R}{\mathbb R} 
\newcommand{\N}{\mathbb N} 
\newcommand{\Z}{\mathbb Z} 
\newcommand{\CC}{\mathbb C} 
\newcommand{\CCC}{\mathscr{C}} 
\newcommand{\F}{\mathcal{F}} 
\newcommand{\Aut}{\operatorname{Aut}} 
\newcommand{\Ad}{\operatorname{Ad}} 
\newcommand{\g}{\mathfrak{g}} 
\newcommand{\p}{\mathfrak{p}} 
\newcommand{\n}{\mathfrak{n}} 
\newcommand{\adiab}{\mathfrak{a}} 

\newcommand{\rr}{\rightrightarrows} 
\newcommand{\act}{\curvearrowright} 
\newcommand{\lag}{\langle}
\newcommand{\rag}{\rangle}
\newcommand{\T}{\mathcal{T}} 
\newcommand{\hd}{\Omega^{1/2}} 
\newcommand{\rk}{\operatorname{rk}} 
\newcommand{\pr}{\operatorname{pr}} 
\newcommand{\ev}{\operatorname{ev}} 
\newcommand{\HN}{\operatorname{HN}} 
\newcommand{\Fb}{\mathcal{F}} 
\newcommand{\gr}{\mathfrak{gr}} 
\newcommand{\U}{\mathbb U} 

\newcommand\isomto{\stackrel{\sim}{\smash{\longrightarrow}\rule{0pt}{0.4ex}}}

\def\dar[#1]{\ar@<2pt>[#1]\ar@<-2pt>[#1]}
\entrymodifiers={!!<0pt,0.7ex>+}

\newcommand{\eq}[1][r]
 {\ar@<-3pt>@{-}[#1]
 \ar@<-1pt>@{}[#1]|<{}="gauche"
 \ar@<+0pt>@{}[#1]|-{}="milieu"
 \ar@<+1pt>@{}[#1]|>{}="droite"
 \ar@/^2pt/@{-}"gauche";"milieu"
 \ar@/_2pt/@{-}"milieu";"droite"}
\everymath{
\everymath{}}

\numberwithin{equation}{section}

\newtheorem{thm}{Theorem}[section]

\newtheorem{lem}[thm]{Lemma}
\newtheorem{prop}[thm]{Proposition}
\newtheorem{cor}[thm]{Corollary}

\theoremstyle{definition}
\newtheorem{mydef}[thm]{Definition}
\newtheorem{ex}[thm]{Example}
\newtheorem{rem}[thm]{Remark}

\begin{document}

\allowdisplaybreaks

\newcommand{\arXivNumber}{2512.11434}

\renewcommand{\PaperNumber}{091}

\FirstPageHeading

\ShortArticleName{Stratification of the Helffer--Nourrigat Cone}

\ArticleName{Stratification of the Helffer--Nourrigat Cone}

\Author{Cl\'ement CREN}

\AuthorNameForHeading{C.~Cren}

\Address{Mathematisches Institut, Georg-August Universit\"at G\"ottingen, \\
Bunsenstra{\ss}e 3-5, D-37073 G\"ottingen, Germany}
\Email{\mail{clement.cren@mathematik.uni-goettingen.de}}
\URLaddress{\url{https://clementcren.github.io/}}

\ArticleDates{Received April 02, 2026, in final form September 09, 2026; Published online September 22, 2026}

\Abstract{Given a~singular filtration on a~manifold, e.g., a~sub-Riemannian setting, one can understand the regularity problems through the Androulidakis--Mohsen--Yuncken pseudodifferential calculus. The principal symbol in this calculus involves the unitary representations of a~family of graded nilpotent groups. Not all the irreducible representations of these groups have to be taken into account, however, the ones that should be considered form the Helffer--Nourrigat cone. This space thus plays the role of a~phase space in sub-Riemannian geometry. Its topology is however very singular, preventing any kind of geometry on it. We propose a~way to desingularize it. The unitary spectrum of a~nilpotent group can be stratified into strata that are locally compact Hausdorff, following Puk\'anszky and Pedersen. We show how this stratification extends to the whole Helffer--Nourrigat cone. As a~byproduct, we show that the $C^*$-algebra of principal symbols and the one of pseudodifferential operators of order~0 are solvable with explicit subquotients.}


\Classification{35H10; 22E25; 22E27; 53C17; 47G30}

\section{Introduction}

Let $X_1, \dots, X_m$ be vector fields on a~manifold $M$ that are Lie bracket generating. This means that for any $x \in M$, the linear span of vectors of the form $X_I(x), I\subset \{1;\dots;m\}^{\N}$ is equal to~$T_xM$, where
\[X_{i_1,\dots,i_N} := [X_{i_1},[\dots,[X_{i_{n-1}},X_{i_N}]\dots].
\]
This condition, known as H\"ormander condition, appeared in~\cite{Hormander}, where it was proved to be a~sufficient condition for the hypoellipticity of the sublaplacian $\sum_{i=1}^m X_i^2$. A~differential operator~$P$ is hypoelliptic if for any distribution $u \in \CCC^{-\infty}(M)$ such that $Pu$ is smooth then $u$ is also smooth.
Given such vector fields, a~natural question is how to determine the hypoellipticity of a~polynomial in the vector fields $X_1, \dots, X_m$, beyond the sublaplacian (i.e., a~sum of squares). A~necessary and sufficient condition for maximal hypoellipticity (a stronger version of hypoellipticity mimicking the usual elliptic estimates) was conjectured in~\cite{HelfferNourrigatPaper} and solved completely in the recent work~\cite{AMYMaximalHypoellipticity}, see~\cite{DebordBourbaki,HelfferNourrigatBook} for a~more complete overview and history of the problem.

To solve this conjecture, the authors of~\cite{AMYMaximalHypoellipticity} develop a~pseudodifferential calculus, and show that the invertibility criterion of the principal symbol (i.e., the ellipticity criterion in this calculus) is exactly what was conjectured by Helffer and Nourrigat. The idea of this new calculus is to consider the vector fields $X_{i_1,\dots,i_N}$ as having order $N$. More precisely, a~differential operator that can be written as a~$\CCC^{\infty}(M)$-linear combination of the vector fields $X_{I}$, $|I|\leq N$ has order at most $N$, its order being the minimal such natural number $N$. Because of the H\"ormander condition, every vector field can be written in terms of the vector fields $X_I$, so every differential operator has a~finite, well-defined order. In order to define a~principal symbol at a~point $x\in M$, the idea is to freeze the coefficients at $x$ and look at a~model operator. In the usual calculus, this operator is a~differential operator with constant coefficients on the tangent space $T_xM$. Here, because of the inhomogeneity in the directions, the model operator becomes a~differential operator on a~graded nilpotent group arising naturally from the anisotropy: the osculating groups~\cite{AMYMaximalHypoellipticity,RothschildStein}. In such a~general setting, these nilpotent groups can have a~dimension bigger than the dimension of $M$, they will typically be universal nilpotent groups following the methods of~\cite{RothschildStein}. To understand the model operator, one uses its image in the unitary representations of the group. Not all these representations make sense for our initial operator on $M$, giving rise to the Helffer--Nourrigat cone. This cone therefore plays the role of a~phase space for the calculus at hand as it replaces the cotangent bundle. Indeed, in this sub-elliptic setting, certain phenomena such as observability and controllability (for associated wave or Schr\"odinger operators) are not captured by the classical phase space, but by this singular ``non-commutative'' phase space, see, e.g., \cite{FermanianLetrouit,Letrouit}

The unitary representations of a~nilpotent group form a~topological space with the Fell topology. It was proved by Kirillov~\cite{Kirillov} that this space is homeomorphic to the space of coadjoint orbits, i.e., the dual of the Lie algebra, quotiented by the coadjoint action and endowed with the quotient topology. As such, this space can be quite pathological, usually being not Hausdorff. We can however stratify this space into some nice parts following Puk\'anszky and Pedersen. On these parts the representations can be continuously parameterized, in particular the corresponding subquotient of the group $C^*$-algebra is an algebra of continuous functions (with values in compact operators). In the case of graded groups, this stratification can be also applied to the spectrum of the algebra of principal symbols on the group. It was shown in~\cite{CrenFilteredCrossedProduct} that one can do this stratification over the whole manifold $M$ when the filtration is equiregular, see Section~\ref{Section:PedersenStratification} below. The strata appearing in the Pedersen decomposition of the space of representations of a~nilpotent group can be identified to semi-algebraic sets in the vector space $\g^*$. Therefore, one can hope to get a~similar stratification for the Helffer--Nourrigat cone and use it to adapt geometric constructions used in the classical setting. For instance, can the notion of Hamiltonian dynamics on the cotangent bundle $T^*M$ make sense on the Helffer--Nourrigat cone? If so, one can hope to understand questions such as propagation of singularities, of semiclassical measures, define a~geometric control condition, study quantum ergodicity (see, e.g., \cite{benedetto}). Such stratification can also be used in index theory as in~\cite{GoffengKuzmin}.

We wish here to extend this stratification in a~general setting. More precisely, given a~family of vector fields on $M$ satisfying H\"ormander's condition, we denote by $\gr(\F)$ the family of osculating groups over $M$ ($\F$ denotes the singular Lie filtration of the tangent bundle induced by the vector fields, see below). We denote by \smash{$\HN(\F) \subset \widehat{\gr(\F)}$} the Helffer--Nourrigat cone, it is a~closed subset of the irreducible unitary representations of all the osculating groups. For $x \in M$, the fiber of the cone at $x$ contains the trivial representation of the osculating group $\gr(\F)_x$. We consider $M$ as a~subset of $\HN(\F)$ through the trivial representations of the fibers and denote by $\HN_0(\F)$ the complement of this inclusion. The spaces $\HN( \F)$ and $\HN_0(\F)$ are endowed with an action of~$\R^*_+$ obtained from the inhomogeneous dilations of the underlying osculating groups. This action is free and proper on $\HN_0(\F)$. We prove the following theorem.

\begin{thm}[stratification of the Helffer--Nourrigat cone]\label{Theorem:Stratification}
 Let $\Fb$ be a~finitely generated singular Lie filtration on a~manifold $M$. There exists a~filtration by open sets
 \[\varnothing = U_0 \subset U_1 \subset \cdots \subset U_d = \HN(\Fb),\]
 such that all the spaces $\Omega_j = U_j\setminus U_{j-1}$ are locally compact Hausdorff.

 We also have \smash{$\HN_0(\Fb) = \bigsqcup_{j=1}^{d-1} \Omega_j \sqcup (\Omega_d \setminus M)$}.
 The action of $\R^*_+$ on $\HN(\Fb)$ induced by the inhomogeneous dilation on $\gr(\Fb)$ preserves these strata. The spaces \smash{${\Omega_j}/{\R^*_+}$}, $1\leq j < d$ and~\smash{${(\Omega_d\setminus M)}/{\R^*_+}$} are locally compact Hausdorff as well.

 Finally, for $x\in M$, denote by $\Omega_{j,x} = \Omega_j\cap \HN(\Fb)_{x}$. Then, if $x \in M_{\rm reg}$, the non-empty sets in the collection $(\Omega_{j,x})_{1\leq j \leq d}$ form a~Pedersen stratification of the group $\gr(\Fb)_x$. If $x \in M \setminus M_{\rm reg}$, the sets $\Omega_{j,x}$ are closed subsets of the Pedersen strata of some Pedersen stratification of $\gr(\Fb)_x$.
\end{thm}

This theorem shows that one can take a~Pedersen stratification of each of the osculating groups and ``glue'' the strata together into a~stratification of the Helffer--Nourrigat cone, while retaining reasonable topological properties. There are however some subtleties to that construction, some of which are showcased in the examples of Section~\ref{Section: ExamplesCones}.

We can then use this result on the algebra of principal symbols of~\cite{AMYMaximalHypoellipticity}. Let $\mathcal{F}^1 = \langle X_1,\dots,X_m\rangle$ and $\mathcal{F}^{j+1} = \mathcal{F}^j + \big[\mathcal{F}^j,\mathcal{F}^1\big]$, $j\geq 1$. We see that they are (locally) finitely generated sub-modules of vector fields and satisfy the condition $
\forall i,j\geq 1$, $ \big[\F^i,\F^j\big]\subset \F^{i+j}$.
More generally, we call a~collection of sub-modules of vector fields with these properties a~singular Lie filtration on $M$. Following~\cite{AMYMaximalHypoellipticity}, one can create spaces of pseudodifferential operators~${\Psi_{\Fb}^k(M)}$, $k \in \Z$. For $k = 0$, these operators act continuously on $L^2(M)$ and one can consider their closure $\Psi_{\Fb}(M)$. They obtain the exact sequence
\[\xymatrix{0 \ar[r] & \mathcal{K}\bigl(L^2(M)\bigr) \ar[r] & \Psi_{\Fb}(M) \ar[r] & \Sigma_{\Fb}(M) \ar[r] & 0.}\]

The algebra $\Sigma_{\Fb}(M)$ is the corresponding $C^*$ algebra of principal symbols. We prove the following result, which is an extension of a~theorem of Fermanian and Fischer~\cite{FermanianFischerGradedLie}.

\begin{thm}\label{Theorem:SpectrumSymbols}
 Let $\Fb$ be a~singular Lie filtration on a~manifold $M$. Then the spectrum of the algebra of principal symbols $\Sigma_{\Fb}(M)$ is naturally homeomorphic to \smash{${\HN_0(\F)}/{\R^*_+}$}.
\end{thm}

Combining this result with Theorem~\ref{Theorem:Stratification}, we obtain the following.

\begin{cor}\label{Corollary:Solvability}
 Let $\Fb$ be a~singular Lie filtration on a~manifold $M$. Then the algebra of principal symbols $\Sigma_{\Fb}(M)$ and the algebra of pseudodifferential operators of order $0$ $\Psi_{\Fb}(M)$ are $($explicitly$)$ solvable.
\end{cor}

The paper is organized as follows. Section~\ref{Section: Helffer--Nourrigat cone} gives the construction of the Helffer--Nourrigat cone. We give two different construction, one intrinsic, using the osculating groups, following~\cite{AMYMaximalHypoellipticity}. The other one uses an additional graded group that surjects onto all the osculating groups. This extrinsic construction goes back to the work of Rothschild and Stein~\cite{RothschildStein} and is the original construction of Helffer and Nourrigat~\cite{HelfferNourrigatPaper}. We also prove that the two constructions yield homeomorphic cones. Section~\ref{Section:PedersenStratification} then describes the Pedersen stratification of a~group. Section~\ref{Section: ExamplesDual} provides examples of such stratification on groups. Section~\ref{Section : Stratification HN Cone} then extends Pedersen's stratification to the Helffer--Nourrigat cone. We prove Theorem~\ref{Theorem:Stratification} reformulated in Theorem~\ref{Theorem:Stratification-InText}. Section~\ref{Section: ExamplesCones} then gives examples of such stratification both in equiregular and singular situations. Section~\ref{Section: Pseudos} then explains some features of the Androulidakis--Mohsen--Yuncken calculus, before proving Theorem~\ref{Theorem:SpectrumSymbols}.

\subsection*{Conventions and notations} All the Lie algebras involved in this paper are nilpotent. Because of this, we will use the Baker--Campbell--Hausdorff formula to identify the Lie algebra with the connected, simply connected group integrating it. We will write both of them using lowercase gothic letters (e.g.,~$\g$). The only exception to that rule are some examples of nilpotent groups in Section~\ref{Section: ExamplesCones}, where the notations try to match existing ones in the literature. The notation $\widehat{\g}$ will refer to the space of irreducible unitary representations of the group $\g$ with its Fell topology. We will denote by $0 \in \widehat{\g}$ the trivial representation.

The letter $\CCC$ will be used to denote the space of continuous functions on a~locally compact, Hausdorff space. The subscript $\CCC_c$ will denote the subspace of compactly supported functions and $\CCC_0$ continuous functions vanishing at infinity. The superscript $\CCC^{\infty}$ will denote smooth functions and will be combined with the aforementioned subscripts.

\section{The Helffer--Nourrigat cone} \label{Section: Helffer--Nourrigat cone}

\subsection{Generalities on distributions and singular Lie filtrations}

In this section, we describe the Helffer--Nourrigat cone of representations. We describe it in two different ways. The first is intrinsic and obtained from the so called collection of osculating groups over each point of the manifold following~\cite{AMYMaximalHypoellipticity}. The second follows the method of Rothschild and Stein~\cite{RothschildStein} and describes the cone as a~subset of representations of a~single graded group (parameterized by the manifold~$M$). This graded group depends on choices of generators for each module $\F^j$, $j\geq 1$, of the filtration. Once these choices are made, they also induce a~map from this group to each osculating group. The dual of this map induces a~homeomorphism between the two constructions of the Helffer--Nourrigat cone.

We start this section by generalities on modules of vector fields. Let $\mathcal{D}\subset \mathfrak{X}(M)$ be a~$\CCC^{\infty}(M)$-sub-module, locally finitely generated. For a~given $x \in M$ we denote by \smash{$\mathcal{D}_x := {\mathcal{D}}/{I_x\mathcal{D}}$} the fiber at $x$, where $I_x= \{ f\in \CCC^{\infty}(M),\, f(x) = 0\}$ is the ideal of functions vanishing at $x$. This vector space is finite-dimensional because $\mathcal{D}$ is locally finitely generated. The evaluation at $x$ induces a~linear map $\mathcal{D}_x \to T_xM$. The image of this map is $D_x = \ev_x(\mathcal{D})$ but the map is not always injective.

The subset $M_{\mathcal{D}} \subset M$ of points over which this map is bijective is called the regular part of $\mathcal{D}$. By~\cite[Proposition~1.5]{AndroulidakisSkandalisHolonomy}, the regular part of $\mathcal{D}$ is an open dense subset of $M$, the family of $D_x, x\in M_{\mathcal{D}}$ forms a~subbundle of $ TM_{\mathcal{D}}$, and the restriction of $\mathcal{D}$ to $M_{|\mathcal{D}}$ is its module of sections. In~\cite{AndroulidakisSkandalisHolonomy}, this is proved for a~singular foliation but the proof of this result only uses the locally finitely generated assumption, so it is also valid in our context.

Finally, denote by $\mathcal{D}^* = \bigsqcup_{x\in M}\mathcal{D}_x^*$. We define a~topology on $\mathcal{D}^*$ by requiring that the projection on $M$ is continuous as well as the maps $\lag \cdot, X \rag, X\in \mathcal{D}$. Following~\cite[Proposition~2.10]{AndroulidakisSkandalisPseudoDiff}, it is a~locally compact Hausdorff space.

\begin{mydef}
 Let $M$ be a~smooth manifold. A~singular Lie filtration is a~family \smash{$\Fb = \bigl(\F^j\bigr){}_{j\geq 1}$} of locally finitely generated $\CCC^{\infty}(M)$-sub-modules of the module of vector fields $\mathfrak{X}(M)$. They satisfy the condition on the Lie brackets $\forall i,j\geq 1$, $\big[\F^i,\F^j\big] \subset \F^{i+j}$.
 Such a~filtration satisfies the H\"ormander condition if $\F^k = \mathfrak{X}(M)$ for $k$ big enough.
\end{mydef}

The H\"ormander condition is technically stronger than the original one. This condition requires that for every $x\in M$, $F^k_x = T_xM$ for $k$ big enough (that might depend on $x$). We can always take a~uniform $k$ locally. Indeed, take $X_1,\dots,X_N \in \F^k$ such that their evaluations~${X_1(x),\dots,X_N(x)}$ at $x\in M$ span $T_xM$. Then there exists an open neighborhood~${x\in U\subset M}$ such that $X_1, \dots, X_N$ is a~basis of $\F^k_{|U}$. In particular, if $M$ is compact, then we can take a~uniform $k$ that works for every point. In the case where the manifold is non-compact, we want to avoid situations where we need more and more depth in the filtration when going to infinity. Therefore, in this paper, we always assume such $k$ to be uniform for all the points in~$M$ and the H\"ormander condition will refer to the definition above, which is slightly stronger than the one in the literature.

\begin{ex}
 Let $X_1, \dots, X_N$ be vector fields on $M$. Set $\F^1 = \CCC^{\infty}(M)\lag X_1,\dots,X_N\rag$, and define inductively
 \[\F^{k+1} = \F^k + \big[\F^k,\F^1\big].\] Then $\F = \bigl(\F^k\bigr)_{k\geq 1}$ is a~singular Lie filtration. The H\"ormander condition on the filtration is then equivalent to the one on the vector fields (from the initial work of H\"ormander~\cite{Hormander}) when $M$ is compact (in general, the weaker version given above is equivalent to the H\"ormander condition on the vector fields).
\end{ex}

\begin{ex}\label{Example : Finitely generated}
 More generally, one can take vector fields $X_1, \dots, X_N$ and associate weights $\nu_1,\dots,\nu_N \in \N$ to them. Then declare that a~vector field $X_I$, $I\in \{1,\dots,N\}^k$, $k\in \N$ has weight~\smash{$\sum_{j\in I}\nu_j$}. Consider then $\F^k$ to be the sub-module of vector fields generated by the ones of weight at most $k$. The family $\F = \bigl(\F^k\bigr)_{k\geq 1}$ forms a~singular Lie filtration on $M$ and the (weaker version of the) H\"ormander condition on $\F$ again corresponds to the one on the vector fields $X_1, \dots, X_N$.

 This singular Lie filtration is (globally) finitely generated. Conversely, every finitely generated singular Lie filtration is of this form. This can be seen by taking a~basis of the first stratum and iteratively adding vector fields not obtainable by Lie brackets to get a~generating set for each~$\F^k$,~$k\geq 1$.
\end{ex}

\subsection{The Helffer--Nourrigat cone: Intrinsic definition}

We now construct the osculating groups associated to a~singular Lie filtration $\Fb$. At a~given point $x\in M$, we define
\[\gr_x(\Fb):= \bigoplus_{j\geq 1} \frac{\F^j_x}{\F^{j-1}_x} = \bigoplus_{j\geq 1} \frac{\F^j}{\F^{j-1} + I_x\F^j},\]
with the convention $\F^0 = \{0\}$.
In the quotient \smash{$\frac{\F^j_x}{\F^{j-1}_x}$}, the map $\F^{j-1}_x \to \F^j_x$ comes from the inclusion between the modules, $\F^{j-1}\subset \F^j$, but is not necessarily injective between the fibers over~$x$. Therefore, the dimension of $\gr_x(\Fb)$ can be greater than the dimension of~$M$.
Given~${i,j\geq 1}$, $x\in M$, and sections $X\in \F^i$, $Y\in \F^j$, the value of $[X,Y](x) \mod \F_x^{i+j-1}$ only depends on~${X(x) \mod \F_x^{i-1}}$ and $Y(x) \mod \F_x^{j-1}$. This defines a~Lie bracket on $\gr_x(\Fb)$, giving it the structure of a~graded nilpotent Lie algebra. Using the Baker--Campbell--Hausdorff formula, this defines a~connected, simply connected, graded Lie group. We can therefore look at the unitary representations of this group, that we understand through Kirillov's orbit method~\cite{Kirillov}
\[\widehat{\gr_x(\Fb)} \cong {\gr_x(\Fb)^*}/{\Ad^*(\gr_x(\Fb))}.\]
To define the Helffer--Nourrigat cone, we need the adiabatic foliation from~\cite{AMYMaximalHypoellipticity}
\[\adiab\F := \sum_{j\geq 1} t^j\widetilde{\F}^j \subset \mathfrak{X}(M\times \R_+).\]
Here, $t$ denotes the projection onto the second factor and $\widetilde{\F}^j$ is the module of vector fields on~${M\times \R_+}$ generated by the elements $(X,0), X \in \F^j$. By the properties of a~singular Lie filtration, this defines a~singular foliation on $M \times \R_+$. For $x\in M$, the point $(x,0)$ is stationary for $\adiab\F$, therefore the fiber $\adiab\F_{(x,0)}$ is a~nilpotent Lie algebra. It is shown in~\cite[Proposition 1.6]{AMYMaximalHypoellipticity} that this Lie algebra is naturally isomorphic to $\gr_x(\Fb)$. For $t \neq 0$, we have $\adiab\F_{(x,t)} = T_xM$ because of the H\"ormander condition.

On the graded Lie algebra $\gr(\F)_x$, there is a~natural $\R^*_+$-action $\delta_{\lambda}, \lambda>0$ given by multiplication by $\lambda^j$ on $\gr(\Fb)_{j,x}$. This is compatible with the Lie bracket and is thus a~Lie algebra automorphism. We denote the Lie group automorphism the same way. One can check that this action is induced from the zoom action defined on $M \times \R_+$ by $\alpha_{\lambda}(x,t) = \bigl(x,\lambda^{-1}t\bigr)$, which preserves the adiabatic foliation.

\begin{rem}
 A~Lie filtration is called regular if all the $\F^j$ correspond to the sections of subbundles $H^j \subset TM$. With a~regular Lie filtration, the osculating groups all have the same dimension as $M$ and form a~Lie groupoid over $M$. With a~singular Lie filtration, we can consider the open dense subset $\bigcap_{j\geq 1} M_{\F^j}$ formed by the intersection of the respective regular sets. On this subset, the singular Lie filtration is a~regular Lie filtration.
\end{rem}

Now to define the Helffer--Nourrigat cone, consider the locally compact space $\adiab\F^*$. We define~\smash{$\T^*_x\F := \overline{T^*M\times \R^*_+} \cap \adiab\F^*_{(x,0)}$}
and identify it to a~subset of $\gr_x(\Fb)^*$. By~\cite[Proposition~1.12]{AMYMaximalHypoellipticity}, this set is stable under the coadjoint action so we can consider the corresponding subset of representations through Kirillov theory.

\begin{mydef}\label{Def: ConeHF Intrinsic}
 The Helffer--Nourrigat cone is the set
 \[\HN(\Fb) := {\T^*\F}/{\Ad^*(\gr(\Fb))} \subset \widehat{\gr(\Fb)}.\]
 We also define its subset
 \[\HN_0(\Fb) := {(\T^*\F\setminus (M\times \{0\}))}/{\Ad^*(\gr(\Fb))},\]
 which corresponds to removing the trivial representation of each osculating group from the Helffer--Nourrigat cone.
\end{mydef}

\begin{rem}\label{Remark : One cannot use all the representations}
 Consider the simple case where $M = \R$ and define a~filtration of depth 3 by defining $\F^1$, $\F^2$, $\F^3$ as generated by $x^2\partial_x$, $x\partial_x$, $\partial_x$, respectively. The singular part is $\{0\}$. The osculating group at that point is the abelian group $\R^3$ with each component having grading~1,~2 and~3 (generated by the respective equivalence classes of $x^2\partial_x$, $x\partial_x$, $\partial_x$).

 We get a~differential operator of order 4 by setting $D= \bigl(x^2\partial_x\bigr)\partial_x - (x\partial_x)^2$. One can then ``freeze the coefficients'' at $x = 0$. The osculating Lie algebra is generated by the equivalence classes $\big[x^2\partial_x\big]$, $[x\partial_x]$, $[\partial_x]$. The operator $D$ yields the left-invariant operator on the osculating group $\big[x^2\partial_x\big][\partial_x] - [x\partial_x]^2$. Since the osculating group is abelian, we can take the Fourier transform of this operator. We obtain the function $\xi \mapsto \xi_1\xi_3 - \xi_2^2$ in the dual coordinates. It turns out however, after a~simple computation, that $D = -x\partial_x$ and has thus order~2. Its principal symbol, as an operator of order~4, would then have to vanish. Therefore, if we want a~reasonable notion of principal symbol, we can only take representations (which here corresponds to evaluations of the Fourier transform above) corresponding to $\xi \in \R^3$ with $\xi_1\xi_3 - \xi_2^2 = 0$. This turns out to be the Helffer--Nourrigat cone (at $x=0$) in that particular case.
\end{rem}

Since the definition of $\T^*\F$ is invariant under the zoom action, we get the following.

\begin{prop}
 The set $\T^*\F$ is invariant under $\delta^*_{\lambda}$, $\lambda>0$. Consequently, so are the sets $\HN(\Fb)$, $\HN_0(\Fb)$. The action on the quotient is still denoted by $\delta$.
\end{prop}
\begin{proof}
 The adiabatic foliation is preserved by the zoom action, and so is its dual with the dual action. Therefore, at $t=0$, we get the $\R^*_+$-action $\delta_{\lambda}, \lambda > 0$ on $\gr(\Fb)$, and both $\T^*\F$ and $\HN(\Fb)$ are stable under the dual action.
\end{proof}

\subsection{The Rothschild--Stein method}\label{Subsection:Rothschild--Stein}

We now give an alternate construction of the Helffer--Nourrigat cone based on the work of Rothschild and Stein~\cite{RothschildStein} that corresponds to the initial definition of Helffer and Nourrigat~\cite{HelfferNourrigatPaper}. The idea is to choose generators of each $\F^j$ to construct a~bigger graded Lie group that surjects onto every osculating group. Dualizing this projection, we obtain an embedding of the unitary representations of each osculating group into the representations of this bigger group. We can then give an alternate description of the Helffer--Nourrigat cone, where the group does not depend on the point in $M$. We now assume the modules of vector fields are (globally) finitely generated and use the following notations: let $N$ denote the depth of the filtration, i.e., $\F^j = \mathfrak{X}(M)$, $\forall j \geq N$. Now consider as in Example~\ref{Example : Finitely generated} vector fields $X_1,\dots,X_m \in \mathfrak{X}(M)$ and weights $\nu_1,\dots,\nu_m \in \N$ such that the submodule $\F^k$ is generated by the $X_I$ of weight at most~$k$, $I \subset \{1,\dots,m\}^n$, $n\in \N$. Now let $\mathfrak{g}$ be the universal graded Lie algebra of depth $N$ generated by elements $Y_k$, $1\leq k \leq m $ with $Y_k$ having the same degree as~$X_k$, i.e.,~$\nu_k$.

\begin{rem}
 In the original approach of~\cite{RothschildStein}, they only choose a~generating family for $\F^1$. This is because they consider the sub-Riemannian situation where $\F^{j+1} = \F^j + [\F^j,\F^1]$, $\forall j \geq 1$, i.e., all the groups appearing are stratified Lie groups.

 It is also worth mentioning that we only need a~generating family for the vector fields. In particular, given two different generating families, one can consider the union of both of these families. This gives rise to 3 universal nilpotent Lie algebras, $\g_1$, $\g_2$, and $\g$, respectively, and $\g$ projects on both $\g_1$ and $\g_2$. It is not the algebraic sum though, as it would (in general) not be a~universal Lie algebra.
\end{rem}

We define a~linear map $\beta \colon \g \to \mathfrak{X}(M)$ by sending $Y_k$ to $X_k$ and every commutator of weight at most $N$ onto the respective commutator of $X_j$'s. This map is not a~Lie algebra morphism because the Lie bracket of vector fields on $M$ with length bigger than $N$ will not necessarily vanish, whereas it automatically does in $\g$.
For $x \in M$, we get a~map $\beta_x := \ev_x \circ \beta \colon \g \to T_xM$. This map is surjective, so by duality we get a~linear injection $\beta_x^* \colon T^*_xM \hookrightarrow \g^*$.

\begin{mydef}\label{Def: ConeHN Extrinsic}
 The (extrinsic) Helffer--Nourrigat cone, at a~point $x\in M$ is the set of elements~${\ell \in \g^*}$ such that there exists a~sequence \smash{$((t_n,x_n,\xi_n))_{n\in \N} \in \bigl(\R^*_+ \times T^*M\bigr)^{\N}$} with $t_n\rightarrow 0$, $x_n\rightarrow x$ and $\delta_{t_n}^*\beta_{x_n}^*(\xi_n) \rightarrow \ell$.
 This set is stable under coadjoint action on $\g^*$ so we see it as a~subset $\HN_{\g}(\Fb)_x \subset \widehat{\g}$. Similarly as in Definition~\ref{Def: ConeHF Intrinsic}, we also define its subset $\HN_{0,\g}(\Fb)_x$, and denote by~$\HN_{\g}(\F)$ and $\HN_{0,\g}(\F)$ the disjoint union of the corresponding sets for $x\in M$.
\end{mydef}

In this definition, $\delta$ denotes the natural $\R^*_+$-action on the graded group $\g$, similarly as for the osculating groups.

\begin{rem}
 In general the only requirements on the group $\g$ should be that it is a~(finite-dimensional) graded Lie group of depth $N$, equipped with a~linear map $\g \to \mathfrak{X}(M)$ sending $\g_j$ to $\F^j$. This map should preserve commutators of length lower or equal to $N$. Finally, every localized map $\g \to T_xM$, $x\in M$ should be surjective.
\end{rem}

To define this version of the Helffer--Nourrigat cone, we had to make choices in order to construct the group $\g$. We now show that Definition~\ref{Def: ConeHN Extrinsic} coincides with Definition~\ref{Def: ConeHF Intrinsic} and thus does not depend on any choice. To do this, first notice that since $\beta$ sends $\g^j$ into $\F^j$ for each~${1\leq j \leq N}$ then it induces a~map $\gr(\beta) \colon \g \to \gr(\Fb)$.
This map sends each $Y_j$ to $X_j \mod \F^{k_j-1}$ where $1\leq k_j\leq N$ is the smallest integer for which $X_j \in \F^{k_j}$. We can also localize this map to get a~collection of maps
\[\gr(\beta)_x \colon\ \g \to \gr(\Fb)_x,\qquad x \in M.\]
All these maps are surjective by construction. Moreover, since $\beta$ preserves Lie brackets of length lower or equal to $N$ then each $\gr(\beta)_x$ is a~Lie algebra homomorphism. Indeed the Lie brackets of length lower or equal to $N$ are preserved, and the ones of higher length are equal to $0$ on both sides. As such, $\gr(\beta)_x$ also induces a~group homomorphism between the respective (connected, simply connected) Lie groups, denoted the same way.

Either via precomposition, or via the transpose map $\tsp{}\gr(\beta)_x $ and the use of the orbit method, the map $\gr(\beta)_x$ induces a~continuous map
\[\widehat{\gr(\beta)_x} \colon\ \widehat{\gr(\Fb)_x} \to \widehat{\g},\qquad x \in M.\]

\begin{prop}\label{Proposition : Independence HN}
 Let $x\in M$, the map \smash{$\widehat{\gr(\beta)_x}$} is a~topological embedding. Moreover, it induces a~homeomorphism between $\HN(\Fb)_x$ and $\HN_{\g}(\Fb)_x$. In particular, the latter does not depend on the choice of generators of $\Fb$ {\rm(}up to homeomorphism$)$. This homeomorphisms also identifies the subsets $\HN_0(\F)$ and $\HN_{0,\g}(\F)$.
\end{prop}

\begin{proof}
 We use the orbit method. Since $\gr(\beta)_x$ is surjective, then $\tsp\gr(\beta)_x\colon \gr_x(\Fb)^* \to \g^*$ is injective. Then if $\xi \in \gr_x(\Fb)^*$, $g \in \g$,
 \[\tsp\gr(\beta)_x\bigl(\Ad^*_{\gr(\beta)_x(g)}\xi\bigr) = \Ad^*_g(\tsp\gr(\beta)_x(\xi)).
 \]
 Therefore, two points $\xi,\xi' \in \gr_x(\Fb)^*$ lie in the same coadjoint orbit if and only if the points $\tsp\gr(\beta)_x(\xi),\tsp\gr(\beta)_x(\xi') \in \g^*$ lie in the same coadjoint orbit. Consequently, the corresponding quotient map \smash{$\widehat{\gr(\beta)} \colon \widehat{\gr_x(\Fb)} \to \widehat{\g}$} is injective and is also a~topological embedding.

 To show that the Helffer--Nourrigat cones are the same, we extend $\beta$ to $\widetilde{\beta} \colon M \times \g \times \R \to TM \times \R$ by $\widetilde{\beta}(x,X,t) = (x,\beta_x(\delta_t(X)),t)$.
 The map $\widetilde{\beta}$ takes values (when applied to sections) in the adiabatic foliation $\mathfrak{a}\F$. The map $\widetilde{\beta}^* \colon \mathfrak{a}\F^*\to M \times \g^* \times \R$ is also an embedding. At~${t=0}$, we obtain the previous map $\tsp\gr(\beta)$. At~${t\neq 0}$, we obtain the map $\delta^*_t\beta^*$. Therefore, since the limit conditions defining $\HN_{\g}(\F)$ and~$\HN(\F)$ are the same, we get
 \[\widehat{\gr(\beta)}(\HN(\Fb)) = \HN_{\g}(\Fb).\tag*{\qed}
 \]\renewcommand{\qed}{}
\end{proof}

\begin{rem}
 The independence of the choice of generators can also be proved without the intrinsic version of the Helffer--Nourrigat cones. If we get two sets of generators, we can construct the respective groups $\g_1$, $\g_2$. We can also create another nilpotent group $\g$ by considering the universal Lie algebra generated by both the generators used for $\g_1$ and $\g_2$. Given some $x \in M$, the map $\g \to \gr(\Fb)_x$ factors through both the maps $\g_i \to \gr(\Fb)_x$. Dually, we get a~commutative diagram
 \[\xymatrix{\widehat{\g} & \ar[l]_{\iota^1_x} \widehat{\g_1} \\
 \widehat{\g_2}\ar[u]^{\iota^2_x} & \ar[u] \ar[l]\widehat{\gr(\Fb)_x}.}\]
 The maps $\iota^i_x$ are both topological embeddings and, even though the images of $\iota_x^1$ and $\iota_x^2$ might differ, we have
 \[\iota_x^1(\HN_{(0),\g_1}(\Fb)) = \iota_x^2(\HN_{(0),\g_2}(\Fb)).\]
 Indeed, both are equal to the image of \smash{$\HN_{(0)}(\Fb)_x$} under the map \smash{$\widehat{\gr(\Fb)_x} \to \widehat{\g}$}.
\end{rem}

From now on, we will omit the subscript $\g$ in the Helffer--Nourrigat cone and will not make a~distinction between its intrinsic and extrinsic definition.

\section{Pedersen stratification of a~group}\label{Section:PedersenStratification}

Let $\g$ be a~graded Lie algebra. Recall that a~Jordan--H\"older basis is a~basis $Y_1, \dots, Y_n$ of $\g$ seen as a~vector space, such that every $\g_j = \lag Y_1,\dots,Y_j\rag$ is an ideal of $\g$. The idea of the Pedersen stratification is to regroup elements $\xi \in \g^*$ based on the numbers
\[d_{j,i}(\xi) := \dim\bigl(\Ad^*(\g_i) \bigl(\xi \mod \g_j^{\perp} \bigr)\bigr).
\]
The subspaces $\g_j$ being ideals, they are stable under the adjoint action. Therefore, $\g_j^{\perp}\subset\g$ is stable under the coadjoint action. Therefore, we get an action of $\g$ (hence of each $\g_i$) on the quotient \smash{${\g^*}/{\g^{\perp}_j}$}. We can see $\g_j^*\subset \g^*$ because we have chosen a~basis of $\g$, we then have~\smash{$\g_j^* \cong {\g^*}/{\g^{\perp}_j}$} and $d_{j,i}$ is the dimension of the orbit under $\g_i$ of the projection of $\xi$ onto $\g_j^*$.

The following properties of these numbers can be found in~\cite{Bonnet,PedersenStratification}.

\begin{prop}
 Let $\xi \in \g^*$, the numbers $d_{j,i} = d_{j,i}(\xi)$, $1\leq i,j\leq n$, satisfy
 \begin{enumerate}\itemsep=0pt
 \item[$(1)$] $d_{i,j} = d_{j,i}$.
 \item[$(2)$] $d_{j,i} - d_{j-1,i} \in \{0;1\}$.
 \item[$(3)$] $d_{j,j} - d_{j-1,j-1} \in \{0;2\}$.
 \item[$(4)$] If $d_{j,i} = d_{j-1,i}$, then $d_{j,k} = d_{j-1,k}$ for all $1\leq k \leq i$.
 \item[$(5)$] For $1\leq i \leq n$, $d_{i,i} = d_{i-1,i}+1$ if and only if $d_{i,i} = d_{i-1,i-1}+2$.
 \item[$(6)$] For all $g \in \g$, $d_{j,i}(\Ad^*_g(\xi)) = d_{j,i}(\xi)$.
 \end{enumerate}
\end{prop}

Before proving the result, notice that the last property implies that the numbers $d_{j,i}(\xi)$ are invariants of the coadjoint orbit corresponding to $\xi$. The idea of Pedersen's stratification is to regroup all the points (hence all the orbits) in $\g^*$ that have the same $d_{j,i}$. However, we also want to order them to get not just a~union of orbits but a~stratification.

\begin{proof}
 We can compute the dimension of the orbit by computing the dimension of its tangent space. Let $M(\xi)$ be the matrix in the Jordan--H\"older basis of the bilinear map
 $(X,Y) \mapsto \lag\xi,[X,Y]\rag$.
 Let $M_{(j,i)}$ be the sub-matrix of $M(\xi)$ formed by the $j$-th first rows and $i$-th first columns. Then $d_{j,i} = \rk(M_{(j,i)}(\xi))$. Since $M(\xi)$ is antisymmetric, this gives (1).

 The number $d_{j,i}-d_{j-1,i}$ is the dimension of the orbit under $\g_i$ of the image of $\xi$ in \smash{${\g_j^*}/{\g_{j-1}^*}$}.~This space being of dimension 1, the dimension of the orbit is either $0$ or $1$. This proves item~(2).

 The matrix $M(\xi)$ being antisymmetric, the rank $d_{j,j}(\xi)$ of the sub-matrices $M_{(j,j)}(\xi)$ (still antisymmetric) can only increase by $2$ (or not change at all) when we go from $j-1$ to $j$. This proves (3).

 If $d_{j,i} = d_{j-1,i}$, this means that the projection of $\xi$ in ${\g_j^*}/{\g_{j-1}^*}$ is a~fixed point for the action of~$\g_i$ (the orbit has dimension $0$). Then it is also the same for the subgroups $\g_k$, $1\leq k \leq i$. This~proves~(4).

 The matrix $M_{(i,i-1)}$ is an intermediate sub-matrix between $M_{(i-1,i-1)}$ and $M_{(i,i)}$ so (5) follows directly from the previous points by considering the contrapositions.

 Finally, let $g\in \g$, we have
 \begin{align*}
 \Ad^*(\g_i)\Ad^*_g(\xi) &= \Ad^*_g\Ad^*_{g^{-1}}\Ad^*(\g_i)\Ad^*_g(\xi) = \Ad^*(g)\Ad^*\bigl(g\g_ig^{-1}\bigr)\xi= \Ad^*(g)\Ad^*(\g_i)\xi.
 \end{align*}
 Since $\Ad^*_g$ is a~diffeomorphism of $\g^*$, it preserves the dimension and we get (6).
\end{proof}

\begin{ex}
 Let $\g = \g_1\oplus\cdots\oplus \g_N$ be a~graded Lie algebra. Let $Y_1, \dots, Y_n$ be a~basis of~$\g$ of eigenvectors for the inhomogeneous dilations. We have $\delta_{\lambda}(Y_i) = \lambda^{\nu_i}Y_i$, $\lambda>0$, $1\leq 1 \leq n$. If the $(\nu_i)_{1\leq i \leq n}$ form a~non-increasing sequence then $(Y_i)_{1\leq i \leq n}$ is a~Jordan--H\"older basis. With a~graded Lie algebra, the weights $\nu_i$ are integers and correspond to the grading of the elements~$Y_i$, i.e., $Y_i \in \g_{\nu_i}$. The way to obtain such a~basis is to first take a~basis of $\g_N$, then complete it with a~basis of $\g_{N-1}$, and so on until we have a~basis of $\g$.

 We call such a~Jordan--H\"older basis compatible with the grading.
\end{ex}

We now consider the indices for which a~jump in dimension happens. We define
\[J_i(\xi) = \{1\leq j \leq i,\, d_{j,i} = d_{j-1,i} +1 \},\qquad 1\leq i \leq n.\]
We can reconstruct the numbers $d_{j,i}(\xi)$ from these sets using properties (1)--(5). We put a~partial order on subsets $I, I' \subset \{1;\dots;n\}$ by $I\prec I' \Leftrightarrow \min(I\setminus I') \leq \min(I' \setminus I)$.
With the convention~${\min(\varnothing) = +\infty}$, this defines a~total ordering for which the maximal element is $\varnothing$.

Now the collection of all our orbits form a~set $\mathcal{J}(\xi)= (J_n(\xi),\dots,J_1(\xi)) \in \mathcal{P}(\{1,\dots,n\})^n$.

Notice that using (6), we have that for all $\xi \in \g^*$ and $g\in \g$, $\mathcal{J}(\Ad^*_g(\xi)) = \mathcal{J}(\xi)$.

We put the lexicographic order on $\mathcal{P}(\{1,\dots,n\})^n$ using the previous order $\prec$ on each component.

\begin{mydef}[Pedersen stratification]
 Let $\g$ be a~nilpotent group with a~choice of Jordan--H\"older basis. For $\mathcal{J}_0 \in \mathcal{P}(\{1,\dots,n\})^n$ we set $\Lambda_{\mathcal{J}_0} = \{\xi \in \g^* \mid \mathcal{J}(\xi) = \mathcal{J}_0\}$. The non-empty sets obtained that way can be ordered using the previous ordering for the indexing $\mathcal{J}_0$. We denote by $\Lambda_1, \dots, \Lambda_d$,
 the subsets obtained and ordered that way. We have $\g^* = \bigsqcup_{k = 1}^d \Lambda_k$. Since these sets are stable under the coadjoint action, we also define
 \[\Omega_k =  {\Lambda_k}/{\Ad^*(\g)},\qquad 1\leq k \leq d.\]
 We call the decomposition $\widehat{\g} = \bigsqcup_{k = 1}^d \Omega_k$ the Pedersen stratification of $\widehat{\g}$.
\end{mydef}

\begin{ex}
 Regardless of the Jordan--H\"older basis, we have $\Lambda_{(\varnothing,\dots,\varnothing)} = \Omega_{(\varnothing,\dots,\varnothing)} = [\g,\g]^{\perp}$.
 Indeed, the points on both sides correspond to the fixed point for the coadjoint action. The set~${(\varnothing,\dots,\varnothing)}$ is necessarily the biggest for the previous order.
\end{ex}

\begin{rem}
 Although the final stratum in the stratification does not depend on the choice of the Jordan--H\"older basis, the others do. For instance, the first stratum is always a~subset of the flat orbits (see~\cite{GoffengKuzmin}) but might depend on the Jordan--H\"older basis, see Example~\ref{Ex: Complex Heisenberg group} below.
\end{rem}

\begin{thm}[Pedersen~\cite{PedersenStratification}]\label{Theorem:ClassicPedersecStrat}
 Let $\g$ be a~nilpotent group with a~choice of Jordan--H\"older basis. Let $\Lambda_k$, $\Omega_k$, $1\leq k \leq d$ denote the Pedersen strata, in $\g^*$ and $\widehat{\g}$ respectively. We have:
 \begin{itemize}
\itemsep=0pt
 \item The sets \smash{$W_k = \bigcup_{\ell\leq k} \Lambda_{\ell}$} are Zariski open in $\g^*$. Consequently, the sets $U_k = \smash{{W_k}/{\Ad^*(\g)}} = \smash{ \bigsqcup_{\ell\leq k}\Omega_{\ell}}$ are open in $\widehat{\g}$.
 \item For every $1\leq k \leq d$, there is a~decomposition $\g^* = V_k^S\oplus V_k^T$ such that each orbit in $\Lambda_k$ meets $V_k^T$ exactly once. The corresponding set $C_k = \Lambda_k \cap V_k^T$ is then a~cross-section for the quotient map $\Lambda_k \to \Omega_k$.
 \item There is a~non-singular, birational bijection $\Psi_k \colon C_k \times V_k^S \to \Lambda_k$ such that for each $v\in C_k$, the map $p_k^T \circ \Psi_k(v,\cdot) \colon V_k^S \to V_k^T$,
 is a~polynomial, whose graph is the orbit $\g \cdot v$. Here $p_k^T$ denotes the projection onto $V_k^T$ parallel to $V_k^S$. The map $\pr_1\circ \Psi_k^{-1} \colon \Lambda_k \to C_k$
 is $\Ad^*(\g)$-invariant and the quotient is the inverse bijection of $C_k \isomto \Omega_k$.
 \end{itemize}
\end{thm}

This theorem shows that we can identify the sets $\Omega_k\subset \widehat{\g}$ to a~semi-algebraic set ${C_k \subset \g^*}$ and that, after some birational reparameterization, the set of corresponding orbits $\Lambda_k$ is homeo\-morphic to $C_k \times V_k^S$. This reparameterization is important as the orbits don't have to be flat. The first property is also crucial in what will follow as it will enable us to define ideals in the group $C^*$-algebra. It also gives sense to the term \emph{stratification} since we are adding ``layers'' to an increasing union of open sets in $\widehat{\g}$.

\begin{rem}
 There is another stratification of the orbits due to Puk\'anszky~\cite{Pukanszky}. It is obtained by regrouping points (hence orbits) $\xi \in\g^*$ that share the same set $J_n(\xi)$ and the order is the same as before. Therefore, Pedersen's stratification is finer than the one of Puk\'anszky in the sense that his strata are subsets of strata in the Puk\'anszky's stratification. The former is often called the \emph{fine} stratification while the latter is the \emph{coarse} stratification. The coarse stratification has similar properties as the fine one, the previous theorem is true for the coarse stratification as well. In most simple examples, the two coincide. See Example~\ref{Ex: Complex Heisenberg group} below to witness a~difference.
\end{rem}

The advantage of using the fine stratification instead of the coarse one can be seen through the next result.

\begin{lem}
 Let $\g$ be a~graded nilpotent group and fix a~Jordan--H\"older basis compatible with the grading. The inhomogeneous dilations $\delta_{\lambda} \in \Aut(\g)$, $\lambda > 0$, define an action of the positive real numbers on $\widehat{\g}$ by precomposition. The Kirillov homeomorphism \smash{$\widehat{\g} \cong {\g^*}/{\Ad^*(G)}$} is equivariant with the action of $\bigl(\tsp\delta_{\lambda}\bigr)_{\lambda > 0}$ on $\g^*$.
 Moreover, the sets $\Lambda_k$, $\Omega_k$, $W_k$, $U_k$ from Theorem~{\rm\ref{Theorem:ClassicPedersecStrat}} are all invariant under the respective $\R^*_+$-action.
 Finally, the action $\R^*_+ \act \Omega_k$ is free and proper for~${k < d}$ and so is the action $\R^*_+\act \Omega_d\setminus\{0\}$.
\end{lem}

\begin{proof}
 The equivariance of Kirillov's map is clear from its construction. Let $\xi \in \g^*$, let $\n \subset \g$ be a~subalgebra for which $\lag\xi,[\n,\n]\rag = 0$ and maximal for that property (we call $\n$ a~polarization for the orbit). Then $\xi$ defines a~character $\chi_{\xi,\n}$ of $\n$ by $X \mapsto \exp(2\ii \pi \xi(X))$. The representation corresponding to the orbit $\Ad^*(\g)\cdot \xi$ is then the induced representation from $\n$ to $\g$: $\mathrm{Ind}_{\n}^{\g}(\chi_{\xi,\n})$. This representation can be shown (cf.\ \cite{Kirillov}) to be irreducible, and not depend on the choice of $\xi$ or $\n$, up to unitary equivalence. Here if we fix $\lambda > 0$ and take $\xi$ and $\n$, then the algebra $\delta_{\lambda^{-1}}(\n)$ is a~polarization for $\tsp\delta_{\lambda}(\xi)$. Moreover, $\chi_{\tsp\delta_{\lambda}(\xi),\delta_{\lambda^{-1}}(\n)} = \chi_{\xi,\n}\circ \delta_{\lambda}$.
 We conclude by functoriality of the induction.

 Since the Jordan--H\"older basis is compatible with the grading, we have $\forall 1\leq i, j\leq n$, $ d_{j,i}\bigl(\tsp\delta_{\lambda}(\xi)\bigr) = d_{j,i}(\xi)$,
 and the rest follows.
\end{proof}

\begin{thm}[Lipsman, Rosenberg~\cite{LipsmanRosenberg}]
 Let $\g$ be a~nilpotent group with a~Jordan--H\"older basis and $\Omega_k \subset \widehat{\g}$ any of the Pedersen strata. Then there exists an integer $n \in \N$ such that every representation of $\Omega_k$ can be realized on the space $L^2\bigl(\R^n\bigr)$, with smooth vectors being the Schwartz functions $\mathscr{S}\bigl(\R^n\bigr)$, and the elements of $\mathcal{U}(\g)$ are realized as differential operators with polynomial coefficients. These coefficients have a~rational non-singular dependence in the representation~${\pi\in \Omega_k}$.
\end{thm}

The next theorem uses the group $C^*$-algebra. Given a~group $\g$, it is obtained by considering the algebra $\CCC^{\infty}_c(\g)$ for the convolution product. It acts on $L^2(\g)$ by convolution. The resulting operators are bounded. The $C^*$-algebra $C^*(\g)$ is then the closure of $\CCC^{\infty}_c(\g)$ in the algebra of bounded operators on $L^2(\g)$ (this is, in general, the reduced $C^*$ algebra, there is also a~notion of maximal $C^*$-algebra but for all the groups considered here, they coincide). Its spectrum is homeomorphic to $\widehat{\g}$ and any of its open subsets defines an ideal of $C^*(\g)$, for which the quotient has a~spectrum homeomorphic to the complement of the open set.

\begin{cor}\label{Cor: Solvability group}
 Let $\g$ be a~nilpotent group with a~Jordan--H\"older basis. Let $\Omega_1,\dots,\Omega_d \subset \widehat{\g}$ be the corresponding Pedersen strata and $U_k = \bigcup_{\ell \leq k} \Omega_{\ell}$ their increasing unions.

 Let $J_k \triangleleft C^*(\g)$ be the ideal corresponding to the open set $U_k$. Then we have an increasing sequence of ideals $\{0\} = J_0 \triangleleft J_1 \triangleleft \dots\triangleleft J_d = C^*(\g)$,
 with for each $1\leq k \leq d$ a~homeomorphism
 \[{J_k}/{J_{k-1}} \cong \CCC_0(\Omega_k,\mathcal{K}_k).\]
 The algebra $\mathcal{K}_k$ is the algebra of compact operators on a~separable Hilbert space, of infinite dimension for $k<d$ and of dimension~$1$ for $k = d$.

 Finally, if $\g$ is graded and the Jordan--H\"older basis compatible with the grading, then the ideals $J_k$ are stable under the action of the $\delta_{\lambda}^* \in \operatorname{Aut}(C^*(\g)), \lambda > 0$. The isomorphism with the subquotient is then $\R^*_+$-equivariant.
\end{cor}

It is not clear if such a~result is true if one replaces the fine stratification with the coarse one. The subquotients would have Hausdorff spectrum, they would be continuous fields of $C^*$-algebras over their spectrum. These fields however have no reason to be trivial. This would require the corresponding Dixmier--Douady invariants
\smash{$\delta_{DD}({J_i}/{J_{i-1}}) \in H^3(\Omega_i;\Z)$},
to vanish (see~\cite{DixmierDouady}). It is however unclear if such a~counterexample exists.

\begin{rem}
 It is worth mentioning though that in some cases we can ensure that it is true. For instance, for polycontact groups (a particular case of graded groups of step 2, such as the Heisenberg group), the coarse stratification does provide subquotients that are trivial fields of $C^*$-algebras. The reason is that for these groups, the orbits are either flat or points. It can be shown (see~\cite[Proposition 9.5]{GoffengKuzmin}) that the set $\Gamma_{\g}$ of flat orbits is open and the corresponding ideal $I_{\g} \triangleleft C^*(\g)$ is isomorphic to $\CCC_0(\Gamma_{\g},\mathcal{K}(\mathcal{H}))$. Therefore, if $\g$ is a~polycontact group, then we have an exact sequence
 \[\xymatrix{0 \ar[r] & \CCC_0(\Gamma_{\g},\mathcal{K}(\mathcal{H})) \ar[r] & C^*(\g) \ar[r] & \CCC_0\left( [\g,\g]^{\perp}\right) \ar[r] & 0.}\]
\end{rem}

\section{Examples of Pedersen stratification for groups}\label{Section: ExamplesDual}

In this section, we describe the Pedersen stratification for some graded groups where computations can be made explicit. They will be used in Section~\ref{Section: ExamplesCones} with classical examples from sub-Riemannian geometry.

\begin{ex}[abelian group] \label{Ex:Abelian group}
 If $\g$ is an abelian graded Lie group then $\widehat{\g} = \g^*$ and the Pedersen stratification consists of only one stratum.
\end{ex}

\begin{ex}[Heisenberg group] \label{Ex: Heisenberg group}
 Let $H_{2n+1}$ be the Heisenberg group, $n\in \N$. Its Lie algebra is spanned by elements $X_1, \dots, X_n$, $Y_1, \dots, Y_n$, $Z$ satisfying the relations $[X_i,Y_j] = \delta_{i,j}Z$ and $Z$ spans the center. In the dual basis, the only non-trivial coadjoint actions are $
 \smash{ \Ad^*_{\exp(tX_k)}Z^*} = Z^* + tY_k^*$, $
 \Ad^*_{\exp(tY_k)}Z^* = Z^* - tX_k^*$.
 Therefore, the coadjoint orbits are either the hyperplanes~${\{Z^* = \lambda\}}$, $\lambda \neq 0$ and the points $\{(x,y,0)\}$, $x,y \in \R^n$. We get two strata, the first~${\Omega_1 = \R^*}$ corresponds to infinite-dimensional representations and the second $\Omega_2= \R^{2n}$ corresponds to the characters. Notice that the two strata cannot be glued in a~Hausdorff way since a~sequence of points converging to $0$ in $\Omega_1$ converges to all the points in $\Omega_2$.
\end{ex}

\begin{ex}[complex Heisenberg group]\label{Ex: Complex Heisenberg group}
 Consider $H_{3}^{\CC}$ obtained by complexifying the Lie algebra but then seen as a~real Lie algebra. It is still nilpotent of depth 2. The first stratum has now dimension $4$ and the center dimension $2$. The coadjoint orbits are similar as for the real Heisenberg group. Flat ones corresponding to $\gamma_1Z^* +\gamma_2\ii Z^* + \CC^2$ for $(\gamma_1,\gamma_2)\neq 0$. The other orbits are the other points, each representing a~single orbit. The stratification however, depends on the Jordan--H\"older basis and differs from the real case.

 Consider first the Jordan--H\"older basis $\ii Z$, $Z$, $\ii Y$, $Y$, $\ii X$, $X $. Consider coordinates $\xi = (\gamma_2,\gamma_1,\beta_2,\beta_1,\alpha_2,\alpha_1)$
 in the dual basis. Then when considering the action of the subgroups of~$H_{3}^{\CC}$, we see that
 \[d_{3,5}(\xi) = \begin{cases}
 1,&\gamma_1 \neq 0 ,\\
 0,&\gamma_1 = 0.
 \end{cases}\]
 Therefore, the strata for this Jordan--H\"older basis are given by $
 \Omega_1 = \{\gamma_1\neq 0\} = \R^*\times \R$, $
 \Omega_2 = \R^*$, $
 \Omega_3 = \CC^2$.
 If we change the order in the Jordan--H\"older basis for $\ii Z$, $Z$, $Y$, $\ii Y$, $X$, $\ii X $ instead, then the stratification is similar but $\Omega_1$ then appears as $\{\gamma_2\neq 0\}$. Changing the coordinates for the~4 elements of depth 1 in the Jordan--H\"older basis, we can get any first stratum that has the form~${\CC \setminus L}$ where $L$ is a~real line. The second stratum will then be $L\setminus \{0\}$ and the third stratum will always be the same~$\CC^2$. In this case, the coarse stratification consists of~2~strata: $\CC \setminus0$ and~$\CC$. A~similar analysis can be done for $H_{2n+1}^{\CC}$, and will give 3 strata again.
\end{ex}

\begin{ex}[Engel group] \label{Ex: Engel group}
 Let $\mathbb{E}$ be the Engel group. Its Lie algebra is spanned by elements $X$, $Y$, $Z$, $V$ of respective degree $1$, $1$, $2$, $3$. The Lie bracket is given by $[X,Y] = Z$ and~${[X,Z] = V}$. In the dual basis, the only non-trivial coadjoint actions are
 \begin{gather*}
 \Ad^*_{\exp(uX)}Z^* = Z^* + uY^*,\qquad
 \Ad^*_{\exp(uX)}V^* = V^* + uZ^* + \frac{u^2}{2}Y^* ,\\
 \Ad^*_{\exp(sY)}Z^* = Z^* - sX^*,\qquad
 \Ad^*_{\exp(tZ)}V^* = V^*-tX^*.
 \end{gather*}
 Therefore, in these coordinates, the orbit of a~point $(x,y,z,v) \in \R^4$ is given by
 \[\left\lbrace \left(x-tv-sz,y+uz+\frac{u^2}{2}v,z+uv,v\right),\, t,s,u\in \R\right\rbrace.\]
 If $v \neq 0$, there is a~unique element of the orbit that has its first and third coordinates equal to~0. If $v = 0$, $z \neq 0$, there is a~unique element of the orbit with all the coordinates except the third being equal to 0. If $v=t=0$, the orbit is a~point. These 3 cases correspond to the Pedersen stratification. We get 3 strata $\Omega_1 \cong \R^* \times \R$, $\Omega_2 \cong \R^*$, $\Omega_3 \cong \R^2$.
\end{ex}

The fact that the last two strata for the Engel group correspond to the strata of the Heisenberg group $H_3$ is not a~coincidence. This is because \smash{$H_3 \cong {\mathbb{E}}/{\mathcal{Z}(\mathbb{E})}$}. This relation can be generalized in a~recursive way with the so called model filiform Lie algebras.

\begin{ex}[model filiform Lie algebra]\label{Ex: Filiform} Let $L_{n}$ be the Lie algebra spanned by elements $X$, $Z_0, \dots, Z_n$ with the Lie bracket defined by $[X,Z_k] = Z_{k+1}$, $0\leq k < n$, and the others being zero. We have $L_1 = H_3$, $L_2 = \mathbb{E}$. More generally, we have \smash{$L_n \cong {L_{n+1}}/{\mathcal{Z}(L_{n+1})}$}.

From this relation, we get that the coadjoint orbit for $L_{n}$ containing a~point of coordinate $(\beta,\alpha_0,\dots,\alpha_{n-1},0)$ is the coadjoint orbit for $L_{n-1}$ of the point of coordinate $(\beta,\alpha_0,\dots,\alpha_{n-1})$, embedded in the hyperplane of equation $\alpha_n^* = 0$.

Therefore, by induction, $L_{n}$ has $n+1$ strata $\Omega_1^{L_{n}},\dots,\Omega_{n+1}^{L_{n}}$. They satisfy the recursion relation \smash{$\Omega_k^{L_{n+1}} = \Omega_{k-1}^{L_n}$}, $1<k\leq n+2$. In particular, we get that \smash{$\Omega_{n+1}^{L_{n}} \cong \R^2$}. Also for $1<k<n+1$, \smash{$\Omega_k^{L_{n}} = \Omega_1^{L_{n-k+1}}$}. So now we only need to compute the first stratum, i.e., the orbits for which~${\alpha_n \neq 0}$.
Using
\[
\Ad^*_{\exp(tZ_{n-1})}Z_n^* = Z_n^*-tX^*,\]
 we can make $\beta$ vanish. The only action we can then use is
\[
\Ad^*_{\exp(tX)} \colon\ Z_k^* \mapsto \sum_{j = 0}^k \frac{t^{k-j}}{(k-j)!}Z_{j}^*.\]
Taking $t = -\frac{\alpha_{n-1}}{\alpha_n}$, we can get rid of $\alpha_{n-1}$ and find a~unique point in the orbit of the form
$(0,\alpha_1',\dots,\alpha_{n-2}',0,\alpha_{n})$. Therefore, we get \smash{$\Omega_1^{L_{n}} \cong \R^* \times \R^{n-1}$}.
Combining with the recursion relation, this gives \smash{$\Omega_k^{L_n} = \R^* \times \R^{n-k}$}, $1\leq k \leq n$.
\end{ex}

\section{Pedersen stratification of the Helffer--Nourrigat cone}\label{Section : Stratification HN Cone}

Let $(M,\Fb)$ be a~singular Lie filtration. According to Section~\ref{Section: Helffer--Nourrigat cone}, the Helffer--Nourrigat cone is defined as a~closed subset $\HN(\Fb)$ of the coadjoint orbits of $\gr(\Fb)$. As before, we make the assumption that $\Fb$ is globally finitely generated so we may use the Rothschild--Stein method to find a~global graded Lie basis with model group $\g$. We may then take a~Pedersen stratification of $\widehat{\g} = \bigsqcup_{j=1}^d\Omega_j^{\g}$. Since $\HN(\Fb) \subset M\times \widehat{\g}$ is closed, then all the sets
\[
\Omega_j := \bigl(M\times\Omega_j^{\g}\bigr)\bigcap \HN(\Fb),
\]
are closed subsets of $M \times \Omega_j^{\g}$ hence locally compact and Hausdorff. Moreover, they form a~partition of $\HN(\Fb)$.

Notice however that since we use a~universal Lie algebra construction, it can be fairly big, so some of the Pedersen strata may not intersect $\HN(\Fb)$. We can solve this problem by getting rid of the empty strata and re-indexing the remaining ones. Another thing that may occur is that a~strata $\Omega_j$ could have an empty fiber over a~point $x \in M$ but be non-empty (so have non-empty fiber over other points). See Section~\ref{Section: ExamplesCones} below, for examples, where this phenomenon happens. Typically in a~lot of sub-Riemannian settings, $\F^1$ generates the whole tangent space by itself on an open dense subset $M_{\rm reg}\subset M$ and we only need to consider Lie brackets over the singular set. In that case, the only non-empty stratum over $M_{\rm reg}$ is the last one since it only consists of abelian representations. However, we could have infinite-dimensional representations appearing in addition over the singular set.

The goal of this section is to prove Theorem~\ref{Theorem:Stratification}. The sets $\Omega_j$ defined above define a~stratification of the Helffer--Nourrigat cone. However, given a~point $x \in M$, we want to relate the fibers over $x$ to the Pedersen strata for $\gr(\Fb)_x$. This is a~consequence of the following result.

\begin{lem}
 Let $\g,\mathfrak{h}$ be graded Lie algebras, $\varphi\colon \g\twoheadrightarrow\mathfrak{h}$ a~surjective homomorphism that preserves the grading. Let $Y_1, \dots, Y_N$ be a~Jordan--H\"older basis of $\g$ compatible with the grading. Then there exists a~Jordan--H\"older basis of $\mathfrak{h}$ such that $\widehat{\varphi} \colon \widehat{\mathfrak{h}} \hookrightarrow \widehat{\g}$ maps elements belonging to a~same Pedersen stratum in $\widehat{\mathfrak{h}}$ to the same Pedersen stratum in $\widehat{\g}$. Moreover, the induced map between the Pedersen strata of $\widehat{\mathfrak{h}}$ and $\widehat{\g}$ is injective.
\end{lem}

\begin{proof}
 Define $Z_j = \varphi(Y_j)$, $1\leq j \leq N$. These vectors span the Lie algebra $\mathfrak{h}$. We extract a~basis in the following way. Define $i_1 = \min\{i\mid Z_i \neq 0\}$,
 and then inductively $i_{k+1} = \min\{i \mid Z_i \notin \Span(Z_{i_1},\dots, Z_{i_k})\}$,
 until we get a~basis for some $i_n$.
 The vectors $Z_{i_1}, \dots, Z_{i_n}$ form a~Jordan--H\"older basis of $\mathfrak{h}$, compatible with the grading.

 We now want to relate jump indices for the two bases when we apply $\tsp\varphi$. To do that, we define
$u \colon  \{1,\dots,N\}  \to \{1,\dots,n\}$, $ i \mapsto \max\{k\mid i_k \leq i\}$.
 The map $u$ is a~right inverse to $\sigma \colon k \mapsto i_k$. We have $\varphi(\g_i) = \mathfrak{h}_{u(i)}$ for all $1\leq i \leq N$. Recall also that if $\xi \in \mathfrak{h}^*$ and $g \in \g$ then \smash{$\tsp\varphi\bigl(\Ad^*_{\varphi(g)}(\xi)\bigr) = \Ad^*_g\bigl(\tsp\varphi(\xi)\bigr)$},
 i.e., $\tsp\varphi$ induces a~homeomorphism between the $\mathfrak{h}$-coadjoint orbit of $\xi$ and the $\g$-coadjoint orbit of~$\tsp\varphi(\xi)$.
 Consequently, for all $\xi \in \mathfrak{h}^*$ and $1\leq i \leq N$ we have $J^i\bigl(\tsp(\varphi(\xi)\bigr) = \sigma\bigl(J^{u(i)}(\xi)\bigr)$.
 We therefore map Pedersen strata to Pedersen strata, injectively between the strata.
\end{proof}

\begin{rem}
 The proof of the lemma is for the Pedersen stratification, i.e., the fine stratification. One can see that the same proof also works for the Puk\'anszky stratification, i.e., the coarse stratification.
\end{rem}

Using this result we have proved the following, which is a~reformulation of Theorem~\ref{Theorem:Stratification}.

\begin{thm}\label{Theorem:Stratification-InText}
 Let $(M,\Fb)$ be a~finitely generated singular filtration. We fix generators and let~$\g$ be the group obtained by the Rothschild--Stein method.
 Let $\Omega_j \subset \HN(\Fb)$, $1\leq j \leq d$, be the non-empty sets obtained by intersecting $\HN(\Fb)$ with the Pedersen strata. Then the sets $U_j=\bigcup_{k\leq j}\Omega_k$ are open and $U_d = \HN(\Fb)$, i.e., the sets $\Omega_j$ form a~partition of $\HN(\Fb)$.

 Moreover, if $x \in M_{\rm reg}$, the intersections $\Omega_j\cap \HN(\F)_x$ form a~Pedersen stratification of $\gr(\Fb)_x$. If $x\in M\setminus M_{\rm reg}$ then the intersections $\Omega_j\cap \HN(\F)_x$ are closed subsets of a~Pedersen stratification of $\gr(\Fb)_x$.
\end{thm}

\begin{rem}
 It could happen that the Helffer--Nourrigat cone gives the whole set of representations of the osculating groupoids, even in the singular set. In that case, the intersections~${\Omega_j\cap \HN(\F)_x}$ form a~Pedersen stratification of $\gr(\Fb)_x$ regardless of $x$. See Example~\ref{Ex: Martinet} below for a~case where it happens.
\end{rem}

With this decomposition at hand, we can replicate the result of Corollary~\ref{Cor: Solvability group}. This time the algebra to consider is not directly $C^*(\gr(\Fb))$ as it contains too much representations. Its spectrum is given by \smash{$\widehat{\gr(\Fb)}$} so we can consider the closed subset given by $\HN(\Fb)$. It corresponds to a~quotient of $C^*(\gr(\Fb))$ denoted by $C^*(\T\F)$. This $C^*$-algebra plays a~key role in the work~\cite{AMYMaximalHypoellipticity} where it appears as the right algebra to consider at $t = 0$ in $\mathfrak{a}\F$ to avoid artifacts coming from the singularities of the filtration (which create singularities in the adiabatic foliation). More precisely, some elements of the fiber at $t = 0$ of $C^*(\mathfrak{a}\Fb)$ appear as elements that can be extended to $a \in C^*(\mathfrak{a}\Fb)$ with $\forall t> 0$, $a_t =0$. The $C^*$-algebra $C^*(\T\F)$ is obtained by modding out the ideal of these very discontinuous elements. See~\cite[Theorem 2.20]{AMYMaximalHypoellipticity} and also~\cite{MohsenBlowupFoliation} where it appears in a~more conceptual and general way.

\begin{thm}\label{Theorem:SolvabilityCTF}
 Let $\Omega_1,\dots,\Omega_d\subset \HN(\Fb)$ be the Pedersen strata of the Helffer--Nourrigat cone and $U_k = \bigcup_{j\leq k} \Omega_j$, $k\geq 0$.
 Let $J_k \triangleleft C^*(\T\F)$ be the ideal corresponding to the open subset~$U_k$. Then we get a~sequence of nested ideals $0 = J_0 \triangleleft J_1 \triangleleft \cdots \triangleleft J_d = C^*(\T\F)$,
 such that the subquotients satisfy $\forall k \geq 1$, \smash{${J_k}/{J_{k-1}} \cong \CCC_0(\Omega_k,\mathcal{K}_k)$}.
 The algebra $\mathcal{K}_k$ is the one of compact operators on a~separable Hilbert space, of infinite dimension if $k < d$ and dimension $1$ if $k = d$.
\end{thm}

\begin{proof} We use the global graded Lie basis $\g$ with its Jordan--H\"older basis. Using Corollary~\ref{Cor: Solvability group}, we get ideals $0 = \widetilde{J}_0 \triangleleft \widetilde{J}_1 \triangleleft \cdots \triangleleft \widetilde{J}_D = C^*(\g)$,
 such that the subquotients satisfy
 \[\forall k \geq 1,\qquad {\widetilde{J}_k}/{\widetilde{J}_{k-1}} \cong \CCC_0\bigl(\widetilde{\Omega}_k,\mathcal{K}_k\bigr).\]
 Here the $\widetilde{\Omega}_k, 1\leq k \leq D$ are the Pedersen strata of $\g$.

 The $C^*$-algebra $C^*(\T\F)$ is a~quotient of $\CCC_0(M) \otimes C^*(\g)$ corresponding to the closed subset~${\HN(\F) \subset M \times \widehat{\g}}$. Consider the ideal $J_k \triangleleft C^*(\T\F)$ obtained as the image of $\widetilde{J}_k$ under this quotient map.
 Then \smash{${J_k}/{J_{k-1}}$} is a~quotient of \smash{${\widetilde{J}_k}/{\widetilde{J}_{k-1}}$}. By construction, this quotient corresponds to the closed subset $\Omega_k = \widetilde{\Omega}_k \cap \HN(\Fb)$ so we have
 \smash{${J_k}/{J_{k-1}} \cong \CCC_0(\Omega_k,\mathcal{K}_k)$}.
 We then re-index the ideals $J_k$. Indeed, if there is $k\geq 1$ for which $\Omega_k = \varnothing$, then we have~${J_{k-1} = J_k}$. We can then drop down the ideal $J_k$ and consider $J_{k+1}$ as the next one, reapplying the procedure if~$\Omega_{k+1}$ is also empty \dots\ At the end of this procedure, the only sets~$\Omega_k$ appearing are the non-empty ones, they thus correspond the Pedersen stratification of $\HN(\F)$ as in Theorem~\ref{Theorem:Stratification-InText}.
\end{proof}

\section[Examples of Pedersen stratification for the Helffer--Nourrigat cone]{Examples of Pedersen stratification\\ for the Helffer--Nourrigat cone}\label{Section: ExamplesCones}

In this section, we give examples of Helffer--Nourrigat cones and their Pedersen stratifications. In the equiregular case, we know from INSERT RESULT that the Helffer--Nourrigat cone is the whole unitary dual. The Pedersen stratification then restricts over each point to a~Pedersen stratification of the group corresponding to the fiber.

\begin{ex}[cntact manifolds]
 Let $M$ be a~contact manifold with contact distribution $H$. We set $\F^1 = \Gamma(H)$ and $\F^2 = \Gamma(H) + [\Gamma(H),\Gamma(H)] = \mathfrak{X}(M)$. Then $\gr(\Fb) = H \oplus {TM}/{H}$, the osculating groups are all isomorphic to the Heisenberg group. Using the decomposition from Example~\ref{Ex: Heisenberg group}, we get $
 \Omega_1 = H^{\perp}$, $\Omega_2 = H^*$.
\end{ex}

The next example shows that even in the equiregular case and when all the osculating groups are isomorphic, the Pedersen strata do not necessarily form a~nice fiber bundle over the base manifold.

\begin{ex}[complex contact manifold]
 Let $X$ be a~complex manifold. A~complex contact structure is the data of a~complex subbundle $H\subset TX$ of complex codimension~1. If we write~${\theta \colon TX \to {TX}/{H}=:L}$, then we need the map $\mathrm{d}\theta_H \colon \Lambda^2H \to L$ to be everywhere non-degenerate. We obtain an equiregular Lie filtration of depth 2. The osculating groups form a~bundle of complex Heisenberg groups. The flat orbits in this case are parameterized by $L^*\setminus M$ and the characters by $H^*$. However, in general the vector bundle $L^*$ is not trivial, not even of the form $L_0\otimes \CC$ for a~real line bundle $L_0$. If we take, for instance, the standard complex contact structure on $\mathbb{CP}^{2n+1}$, the corresponding line bundle is $\mathcal{O}_{\mathbb{CP}^{2n+1}}(2)$. This bundle is non-trivial in general and every real line bundle over $\mathbb{CP}^{2n+1}$ is trivial so $L$ cannot be obtained from a~line bundle either.

 Regardless of this issue, we can perform the Pedersen decomposition by choosing local trivializations of $X$ with complex Darboux coordinates. The second stratum will, in general, not be a~real line bundle however (hence the first stratum is not the complement of a~real line bundle either).

 In this case, it is much more comfortable to work with the coarse stratification. With it, we get 2 strata: the flat orbits, identified with $L^*\setminus M$, and the characters, identified with $H^*$.
\end{ex}

\begin{ex}[parabolic homogeneous spaces]
 Let $G$ be a~semi-simple Lie group, $P$ a~parabolic subgroup. We can assume that \smash{$\g = \bigoplus_{j = -k}^k \g_j$} with $[\g_i,\g_j] \subset \g_{i+j}$ for every indices $-{k\leq i,j\leq k}$, and such that \smash{$\p = \bigoplus_{j = 0}^k\g_j$}. Then the homogeneous space \smash{${G}/{P}$} is an equiregular filtered manifold. We have \smash{$T({G}/{P}) = G \times_P {\g}/{\p}$} so the filtration comes from the one of \smash{${\g}/{\p}$} as a~vector space. We have $\gr(\Fb) = G\times_P \g_-$, where \smash{$\g_- = \bigoplus_{j = -k}^{-1}\g_j$} is the opposed nilpotent Lie algebra. The adjoint action of $P$ on $\g$ preserves $\g_-$ and its graduation. Therefore, let~${\Omega_1^{\g_-}, \dots, \Omega_d^{\g_-}}$ be the Pedersen strata of the group $\g_-$. We have $\forall 1\leq i \leq d$, $ \Omega_i = G \times_P \Omega_i^{\g_-}$.
 This example can be extended to manifolds having a~parabolic geometry in the sense of Cartan (i.e., they locally look like this homogeneous space), see, e.g., \cite{DaveHallerBGG}. However, similar problems as in the previous examples could arise in that case and prevent the Pedersen strata to form fiber bundles.
\end{ex}

\begin{ex}[Engel manifold]
 An Engel manifold is a~4-dimensional manifold $M$ endowed with a~distribution $\F^1$ of constant rank $2$ such that $\F^2 = \F^1 + \big[\F^1,\F^1\big]$ has constant rank $3$ and $\F^3 = \F^2 + \big[\F^1,\F^2\big] = \mathfrak{X}(M)$. This is an example of equiregular sub-Riemannian structure.

 For an Engel manifold, all the osculating groups are isomorphic to the Engel group $\mathbb{E}$ and they form a~bundle of groups. There is a~characteristic line field $L \subset F^1$ which is uniquely determined by the property $\big[\Gamma(L),\F^2\big] \subset \F^2$. Taking the notations from Example~\ref{Ex: Engel group}, the characteristic line field on the Engel group would be generated by $Y$. Globalizing the decomposition for the Engel group, we get the strata
 \begin{align*}
 \Omega^1 &= \bigl(F^2\bigr)^{\perp} \oplus (L\setminus M) ,\qquad
 \Omega_2 = \bigl({F^2}/{F^1}\bigr)^{\perp} = \bigl(F^1\bigr)^{\perp_{F^2}} ,\qquad
 \Omega_3 = \bigl(F^1\bigr)^*.
 \end{align*}
\end{ex}

The following examples now present singularities.

\begin{ex}[Martinet distribution] \label{Ex: Martinet}
 Consider the Martinet distribution on $M = \R^3$ with \smash{$\F^1 = \big\lag \partial_x + \frac{y^2}{2}\partial_z,\partial_y\big\rag$}. We have $\F^2 = \F^1 + \lag y\partial_z\rag$ and $\F^3 = \mathfrak{X}\bigl(\R^3\bigr)$. We then have $\gr(\Fb)_{(x,y,z)} = H_3$ if $y\neq 0$ and \smash{$\gr(\Fb)_{(x,0,z)} = \mathbb{E}$}. We have $M_{\rm reg} = \{(x,y,z), y\neq 0\}$ so \smash{$\HN(\Fb)_{(x,y,z)} = \widehat{H_3}$} if $y\neq 0$.
 Take $\mathbb{E} = \Span(X,Y,Z,V)$, and consider the map $\Phi_{p,t} \colon \mathbb{E} \to T_p\R^3$ that sends $X$ to \smash{$t\bigl(\partial_x + \frac{y^2}{2}\partial_z\bigr)$}, $Y$ to $t\partial_y$, $Z$ to $t^2y\partial_z$, $V$ to $t^3\partial_z$ and then evaluates at $p$.
 We have $\ker(\Phi_{p,t}) = \Span((0,0,-t,y))$ with $p=(x,y,z)$. Thus
 \[\im\bigl(\tsp{}\Phi_{p,t}\bigr) = \{\alpha \in \mathbb{E}^*,\, t\alpha_3 -y\alpha_4=0\}.\]
 Therefore, for $y=0$, we can approach $(p,0)$ by a~sequence of points $(p_n,t_n), n\in \N$ with $y_n/t_n$ approaching any given slope. Therefore, we get $\HN(\Fb)_{(x,y,z)} = \widehat{\mathbb{E}}$ when $y=0$.

 Using the Pedersen strata of $\mathbb{E}$ from Example~\ref{Ex: Engel group}, we get
 \begin{align*}
 \Omega_1 &= \{(x,y,z) \in M,\, y=0\}\times \R^*\times \R,\qquad
 \Omega_2 = M \times \R^* ,\qquad
 \Omega_3 = M \times \R^2.
 \end{align*}
\end{ex}

\begin{ex}[$n$-Grushin plane] Consider the $n$-Grushin distribution generated by $\F^1 = \big\lag \partial_x, \frac{x^n}{n!}\partial_y\big\rag$ on $M = \R^2$. We have
\[
\gr(\Fb)_{(x,y)} =\begin{cases}
 \R^2,& x\neq 0,\\
 L_n, & x=0,
\end{cases}
\]
 where $L_n = \lag X,Y_0,\dots,Y_n\rag$ is the model filiform Lie algebra from Example~\ref{Ex: Filiform}. We get linear maps $\Phi_{p,t}\colon L_n \to T_pM$ sending $X$ to $t\partial_x$, $Y_k$ to \smash{$\frac{t^{k+1}x^{n-k}}{(n-k)!}\partial_y$} and evaluating at $p$. We have
 \[\ker(\Phi_{p,t}) = \Span \left( \left( 0,\dots,0,t^{n-k},0,\dots,0,-\frac{x^{n-k}}{(n-k)!}\right),\, 0\leq k \leq n-1\right).\]
 Thus, with dual coordinates, $\beta \in \R$, $\alpha\in \R^{n+1}$, we get
 \[\im\bigl(\tsp{}\Phi_{p,t}\bigr) = \left\lbrace(\beta,\alpha)\in L_n^*,\, \forall 0\leq k < n,\, t^{n-k}\alpha_k = \frac{x^{n-k}}{(n-k)!}\alpha_n\right\rbrace.\]
 As in the previous example, taking a~sequence $(p_n,t_n)$, $n\in \N$ going to $(p,0)$ with $p =(0,y)$, the limit of $x_n/t_n$ determines the limit space in the grassmanian topology. If the limit is infinite, then we obtain the space $\{(\beta,\alpha)\in L_n^*,\, \alpha_n = 0\}$ which corresponds to the union of the $\Omega_i^{L_n}$, $i> 1$. If~${x_n\sim_{n\to \infty}\lambda t_n}$, then the limit space becomes
 \[\left\lbrace (\beta,\alpha)\in L_n^*, \,\forall 0\leq k < n,\, \alpha_k = \frac{\lambda^{n-k}}{(n-k)!}\alpha_n\right\rbrace.\]
 The image in the first Pedersen stratum is characterized by the intersection of the coadjoint orbit with $\{\beta = \alpha_{n-1} = 0\}$. Here this forces all the $\alpha_k$, $0\leq k \leq n-1$ to vanish. Therefore,
 \[\HN(\Fb)_{(0,y)}\cap \Omega^{L_n}_1 = \R^* \times\{0\} \subset \R^* \times \R^{n-1}.\]
 Combining everything we get
 \begin{gather*}
 \Omega_1 = \{(x,y)\in M,\, x=0\} \times \R^* ,\qquad
 \Omega_k = \{(x,y)\in M,\, x=0\} \times \R^* \times \R^{n-k},\\ 1 <k < n+1,\qquad
 \Omega_{n+1} = M \times \R^2.
 \end{gather*}
\end{ex}

In all these examples, the strata are always the same for points in the regular set or its complement. However, we can produce situations where the dimension of the osculating group is not locally constant within the singular set.

\begin{ex}
 Consider $M = \R^3$ with the distribution $\F^1 = \lag \partial_x - xy\partial_z,\partial_y\rag$.
 This time we have two different situations on the singular set
 \[\gr(\Fb)_{(x,y,z)} = \begin{cases}
 H_3,& xy \neq 0,\\
 \mathbb{E},& xy=0,\  (x,y)\neq(0,0),\\
 L_{6,21(-1)},& x=y=0.
 \end{cases}\]
 Here $L_{6,21(-1)} = \Span(X,Y,Z,X',Y',Z')$ satisfies the relations (the terminology is from~\cite{deGraafClassification})
 \begin{gather*}
 [X,Y] = Z,\qquad [X,Z] = X',\qquad [X,Y'] = Z',\qquad
 [Z,Y] = Y',\qquad [X',Y] = Z'.
 \end{gather*}
 This algebra has depth 4. What happens on the singular set away from $(0,0,z)$ is similar to the Martinet case so we have $\HN(\Fb)_{(x,y,z)} = \widehat{\mathbb{E}}$ in that case. For the points $(0,0,z)$, a~quick computation shows that the Helffer--Nourrigat cone is the image of $\{(a,b,c,\alpha,\beta,\gamma)\in L_{6,21(-1)}^*,\, \alpha\beta = c\gamma\}$.
 Let us see to which representations this set corresponds. We first describe the Pedersen strata in the coordinates $(a,b,c,\alpha,\beta,\gamma)$ given by the dual basis.

 The first stratum corresponds to $\gamma \neq 0$, each of these orbits intersect the plane \smash{$\lag Z^*,Z'^*\rag$} exactly once so \smash{$\Omega_1^{L_{6,21(-1)}} = \R^* \times \R$}. The second stratum corresponds to $\gamma =0$, $\beta\neq 0$. All these orbits intersect the space $\lag X^*,X'^*,Y'^*\rag$ at a~single point so \smash{$\Omega_2^{L_{6,21(-1)}} = \R^* \times \R^2$}. The next stratum corresponds to $\gamma = \beta = 0$, $\alpha \neq 0$. All these orbits have a~unique intersection with~${\lag Y^*,X'^*\rag}$ so \smash{$\Omega_3^{L_{6,21(-1)}} = \R^* \times \R$}. The next one corresponds to $\gamma = \beta = \alpha = 0$, $c\neq 0$, all these orbits meet the line $\lag Z^* \rag$ exactly once (at $(0,0,c,0,0,0)$) so \smash{$\Omega_4^{L_{6,21(-1)}} = \R^*$}. The last stratum corresponds to all the points $(a,b,0,0,0,0)$ that are all stationary. Therefore,~\smash{$\Omega_5^{L_{6,21(-1)}} = \R^2$}.
 Now for the Helffer--Nourrigat cone, we need to take the intersection with the cone $\alpha\beta = c\gamma$. When $\gamma \neq 0$, we can take the representative for the orbit of the form $(0,0,c,0,0,\gamma)$. Since $\gamma \neq 0$, this forces~${c=0}$. When $\gamma = 0$, this forces $\alpha = 0$ or $\beta = 0$. This first case intersects the second stratum in a~proper subset and gives the whole fourth and fifth stratum. The second case gives in addition the whole third stratum. Combining all these results, we have
 \begin{gather*}
 \Omega_1 = \{(0,0,z)\in M\} \times \R^*, \qquad
 \Omega_2 = \{(0,y,z) \in M\} \times \R^* \times \R,\\
 \Omega_3 = \{(x,0,z) \in M\} \times \R^* \times \R,\qquad
 \Omega_4 = M \times \R^*,\qquad
 \Omega_5 = M \times \R^2.
 \end{gather*}
\end{ex}
We can see in the last example that some strata are only concentrated on parts of the singular set. To explain this, notice that $L_{6,21(-1)}$ has two different Engel groups as quotient
\[{L_{6,21(-1)}}/{\lag Y',Z'\rag} \cong \mathbb{E} \cong {L_{6,21(-1)}}/{\lag X',Z'\rag}.\]
The first one corresponds to the Engel group appearing as $\gr(\Fb)_{(0,y,z)}$ with $y\neq 0$ and the second one as $\gr(\Fb)_{(x,0,z)}$ with $x\neq 0$.

Remember that for a~given graded group $\g$, the strata Pedersen strata $\Omega_i$ can all be identified to algebraic subsets of $\g^*$. In the previous examples, the intersection with the Helffer--Nourrigat cone stayed algebraic as well. Of course, if the vector fields used to define the filtration are not analytic, this has no reason to stay true. This is however not even the case when the vector fields are analytic.

\begin{ex}[{\cite[Example 1.11.d]{AMYMaximalHypoellipticity}}]
 Let $M = \R^2$ and consider the filtration of step $2$ given by the distributions
 \begin{align*}
 \F^1 &= \bigl\lag x^2y^4\partial_x,x^6\partial_x, x^4y^2\partial_x, y^6\partial_x\bigr\rag,\qquad
 \F^2 = \lag \partial_x,\partial_y\rag.
 \end{align*}
 The osculating groups are all abelian and we have
 \[\gr(\Fb)_{(x,y)} = \begin{cases}
 \R \oplus \R,& (x,y)\neq 0,\\
 \R^4 \oplus \R^2,& (x,y)=(0,0).
 \end{cases}\]
 Since all the groups are abelian, the whole Helffer--Nourrigat cone consists of only one Pedersen stratum. We have
 \begin{align*}
 \HN(\Fb) ={}& \{(x,y)\in M, (x,y)\neq(0,0)\} \times \R^2 \\
 & \sqcup \{(0,0)\}\times \big\{(\xi,\eta)\in \R^5\oplus \R \mid \xi_1^3 = \xi_2\xi_4^2,\, \xi_3^3 = \xi_2^2\xi_4, \, \xi_i\xi_5\geq 0 \,\forall i\big\}.
 \end{align*}
\end{ex}

In the last example, the Helffer--Nourrigat cone is still semi-algebraic. It is not clear if this result holds in general. That is, if the singular Lie filtration is analytic (e.g., $M = \R^n$ and a~filtration given by vector fields with polynomial coefficients), can the Pedersen strata be chosen to be semi-algebraic?

\section{Pseudodifferential operators on filtered manifolds}\label{Section: Pseudos}

In their paper~\cite{AMYMaximalHypoellipticity}, Androulidakis, Mohsen and Yuncken introduce a~class of pseudodifferential operators on a~filtered manifold $(M,\Fb)$. This calculus is made so that sections of $\F^j$ define differential operators of order at most $j$. Moreover, they associate to polynomials in these vector fields a~principal symbol, a~family of operators in the representations of $\HN(\Fb)$. They show that the invertibility of this principal symbol in all the non-trivial representations (i.e., $\HN_0(\Fb)$) characterizes hypoellipticity for the operator. An important step in their proof is to understand the $C^*$-algebra $\Psi_{\Fb}(M)$, obtained by completion of the algebra of operators of order~$0$ in their calculus. This algebra contains the compact operators on $L^2(M)$ and sits in the exact sequence
\begin{equation}\label{Equation:PseudosExactSequence}
 \xymatrix{0 \ar[r] & \mathcal{K} \ar[r] & \Psi_{\Fb}(M) \ar[r] & \Sigma_{\Fb}(M) \ar[r] & 0.}
\end{equation}
The compact operators here appear as the completion of the algebra of negative order operators in the calculus. The quotient map in the exact sequence corresponds to the principal symbol map. At this point, the algebra $\Sigma_{\Fb}(M)$ can be considered as defined by this exact sequence. We will give an explicit description of this algebra as well as the symbol map in Section~\ref{Subsection:PrincipalSymbol} below. The goal of this section is to show Theorem~\ref{Theorem:SpectrumSymbols} which we recall here the following theorem.

\begin{thm}
 Let $(M,\Fb)$ be a~singular Lie filtration. There is a~natural homeomorphism
 \[\widehat{\Sigma_{\Fb}(M)} \cong {\HN_0(\Fb)}/{\R^*_+}.\]
\end{thm}

Given this result, we can use the Pedersen stratification to prove the solvability of $\Sigma_{\Fb}(M)$ and $\Psi_{\Fb}(M)$ (with explicit subquotients), as stated in the introduction.

\subsection{The Androulidakis--Mohsen--Yuncken calculus}

We start by recalling the construction of the pseudodifferential calculus from~\cite{AMYMaximalHypoellipticity}. As in Section~\ref{Subsection:Rothschild--Stein}, and in order to simplify the exposition, we make the assumption that the filtration is finitely generated and denote by $\g$ the universal Lie group generated by a~choice of generators, and $\beta \colon \g \to \mathfrak{X}(M)$ the corresponding linear map.

For $\lambda > 0$, denote by $\alpha_{\lambda}$ the action on $\g \times M \times \R_+^*$ given by $\alpha_{\lambda}(X,x,t) := \bigl(\delta_{\lambda}(X), x,\lambda^{-1}t\bigr)$.

Consider the partially defined map
\begin{align*}
 \ev \colon\ \g\times M \times \R_+^* &\to M \times M \times \R_+ ,\qquad
 (X,x,t) \mapsto (\exp_x(\beta(\delta_{t}(X))),x,t),
\end{align*}
with the convention $\delta_0 = 0$. This map is defined on a~neighborhood of $\g \times M \times \{0\} \cup \{0\} \times M \times \R_+^*$. We can restrain such a~neighborhood to an open set $\mathbb{U}$ still containing the previous set and such that $\ev \colon \mathbb{U} \to M \times M \times \R_+$ is a~smooth submersion and is $\R^*_+$-equivariant (we take $\U$ to be stable under this action). The quadruple $(\g,\beta,\mathbb{U},M)$ is then a~graded basis in the sense of~\cite{AMYMaximalHypoellipticity}. The open set $\mathbb{U}$ and the map $\ev$ serve as a~bi-submersion for the holonomy groupoid of the adiabatic foliation $\mathfrak{a}\F$ near the units. Our assumption on the filtration being finitely generated is equivalent to the existence of a~global graded basis. In full generality, the reasoning of this section can also be applied to any singular Lie filtration by covering the manifold with graded bases.

We now introduce the basic ideas from the calculus of~\cite{AMYMaximalHypoellipticity}. To avoid choosing measures and having scaling factors from the dilations, we work with half-densities.

Consider the maps $r,s \colon \U \to M\times \R_+$ defined by
\begin{align*}
 s(X,x,t) &= (x,t), \qquad
 r(X,x,t) = (\exp_x(\beta(\delta_t(X))),t).
\end{align*}

These maps give $\U$ the structure of a~bi-submersion for the adiabatic foliation $\adiab\F$. One can go further and endow $\U$ with a~local groupoid structure corresponding to the one of the holonomy groupoid of the adiabatic foliation. It interpolates between the pair groupoid structure on $M\times M \times \{t\}$ for $t>0$ and the group structure on the fibers of $\gr(\Fb)$ at $t=0$ (here given through the bigger group $\g$). We then define \smash{$\hd_{r,s} = \Omega^{1/2}\ker(\diff s) \otimes \Omega^{1/2}\ker(\diff r)$} (here $\Omega^{s} E$ denotes the bundle of $s$-densities of a~vector bundle $E$, $s\in \R$). A~quick computation shows that here \smash{$\hd_{r,s} \cong \Omega^1\g$}. If we denote by $\ev_t \colon \U \to M$ the composition of $\ev$ and the evaluation at~${t > 0}$, we get a~submersion, hence
\smash{$\ev_{t*} \colon\ \CCC^{\infty}_c\bigl(\U,\hd_{r,s}\bigr) \to \CCC^{\infty}_c\bigl(M\times M,\hd\bigr)$}.
 Here $\hd$ is the half-densities bundle on the pair groupoid $M \times M \rr M$.

 For a~vector bundle $E\to M$, we denote by $\mathcal{D}'(M,E)$ the topological dual of $\CCC^{\infty}_c\bigl(M,E^*\otimes \Omega^1TM\bigr)$. This way we get the inclusion $\CCC^{\infty}(M,E) \subset \mathcal{D}'(M,E)$.

\begin{mydef}
 For $k \in \Z$, let $\mathcal{E}^{'k}(\U)$ be the subset of distributions $u\in \mathcal{D}'\bigl(\U,\hd_{r,s}\bigr)$ such that%
 \begin{enumerate}
\itemsep=0pt
 \item[(1)] $u$ is transverse to $s \colon \U \to M\times \R_+$ in the sense of~\cite{AndroulidakisSkandalisHolonomy}, which means that if $f \in \CCC^{\infty}\bigl(\U,\Omega^{-1/2}_{r,s}\otimes \Omega^1\ker(\diff s)\bigr)$, then $s_*(fu) \in \CCC^{\infty}(M\times\R_+)$.
 \item[(2)] For any $\lambda>0$, we have \smash{$\alpha_{\lambda*}u-\lambda^ku\in\CCC^{\infty}_c\bigl(\U,\hd_{r,s}\bigr)$}.
 \item[(3)] $u$ is properly supported in time, i.e., $\supp(u)\to \R_+$ is a~proper map.
 \end{enumerate}
\end{mydef}

From the space $\mathcal{E}^{'k}(\U)$, we get two types of maps
\begin{gather*}
 \ev_{t*} \colon\ \mathcal{E}^{'k}(\U) \to \mathcal{D}'\bigl(M\times M,\hd\bigr),\qquad t > 0,\\
 \ev_{x,0} \colon\ \mathcal{E}^{'k}(\U) \to \mathcal{D}'\bigl(\gr(\Fb)_x,\Omega^1\bigr),\qquad x\in M.
\end{gather*}

\begin{mydef}
 The space of pseudodifferential operators of order $k\in \Z$ is defined as
 \[\Psi^k_{\Fb}(M):= \ev_{1*}\bigl(\mathcal{E}^{'k}(\U)\bigr) + \CCC^{\infty}_c\bigl(M\times M,\hd\bigr).\]
\end{mydef}

The definition given in~\cite{AMYMaximalHypoellipticity} is slightly different but coincides with the one given here in the presence of a~global graded basis. In particular, it can be shown to be independent of the choice of a~graded basis (i.e., of the choice of $\g$ here).

\subsection{Principal symbol of an operator}\label{Subsection:PrincipalSymbol}

We now want to define the principal symbol of an operator in the Androulidakis--Mohsen--Yuncken calculus. The idea is to take an operator $P \in \Psi^k_{\Fb}(M)$ then find a~global lift $u \in \mathcal{E}^{'k}(\U)$ (meaning that $\ev_{1*}(u)-P$ is smooth). We then produce from $u$ a~function on $M\times \bigl(\g^*\setminus \{0\}\bigr)$, homogeneous of degree $k$. From this function, we construct a~(family of) multiplier(s) of a~certain subalgebra of $C^*(\g)$ which we can evaluate at the non-trivial representations of $\g$.
The goal is then to make sure that this construction does not depend on any choice. This is not true in general for all the representations, but at least for the ones in the Helffer--Nourrigat cone~$\HN_0(\Fb)$.

The construction of the homogeneous function is a~result due to Taylor~\cite[Proposition 2.4]{Taylor} for a~single graded group, the bundle version below is from~\cite[Proposition 3.4]{AMYMaximalHypoellipticity}.

\begin{prop}
 Let $k\in \Z$ and $u \in \mathcal{E}^{'k}(\U)$. There is a~unique smooth function $A_u \in \CCC^{\infty}\bigl(\g^*\times M \times \R_+ \setminus (\{0\}\times M\times \{0\})\bigr)$ satisfying:
 \begin{enumerate}
\itemsep=0pt
 \item[$(1)$] $A_u$ is homogeneous of degree $k$, i.e.,
 \begin{gather*}
 \forall \lambda >0,\qquad \forall (\eta,x,t) \in \g^*\times M \times \R_+ \setminus (\{0\}\times M\times \{0\}),\\ A_u(\tsp\delta_{\lambda}(\eta),x,\lambda t) = \lambda^kA_u(\eta,x,t).
\end{gather*}
 \item[$(2)$] There exists $K\subset M$ compact such that $\supp(A_u) \subset \g^*\times K\times \R_+$.
 \item[$(3)$] If $\chi \in \CCC^{\infty}_c\bigl(\g^*\times \R_+\bigr)$ is equal to $1$ in a~neighborhood of $(0,0)$ and we set
 \[f(X,x,t) := u(X,x,t) - \int_{\eta\in\g^*}\exp(\ii \lag\eta,X\rag)(1-\chi)(\eta,t)A_u(\eta,x,t),\]
 then $f\in \CCC^{\infty}(\g\times M\times\R_+)$. Moreover, there exists a~compact subset $K \subset M$ and a~positive number $\varepsilon>0$ such that $\supp(f) \subset \g \times K \times [0;\varepsilon]$. Finally, $f$ and all its derivatives in $x$, $t$ are Schwartz on $\g$, uniformly in $x$, $t$.
 \end{enumerate}
 This function is called the full symbol of $u$. Conversely, for every smooth function $A$ satisfying the first two conditions, there is a~distribution $u\in\mathcal{E}^{'k}(\U)$ with $A_u = A$.
\end{prop}

Given such a~total symbol, we can restrict it to the fiber $\g\times M \times \{0\}$ (this corresponds to the total symbol of the original distribution restricted to $\mathbb{U}_{|t=0} = M\times \g$). The resulting function is homogeneous of degree $k$ for the pullbacks by $\tsp\delta$. Let us now work with $A\in \CCC^{\infty}\bigl(\g^*\times M\setminus (\{0\}\times M)\bigr)$, homogeneous of degree $k$. Let $\mathscr{S}_0\bigl(\g, \Omega^1\g\bigr)$ be the algebra of Schwartz functions on $\g$ whose (Euclidean) Fourier transform vanishes at all orders at $0 \in \g^*$. It is an algebra for the group convolution product of $\g$. For a~fixed $x \in M$, we get a~multiplier of $\mathscr{S}_0\bigl(\g,\Omega^1\g\bigr)$
\[f \in \mathscr{S}_0\bigl(\g,\Omega^1\g\bigr) \mapsto \widecheck{A}(\cdot,x) \ast f \in \mathscr{S}_0\bigl(\g,\Omega^1\g\bigr).\]
According to~\cite[Proposition~2.2]{CGGP}, this map is well defined and continuous. We denote by $\sigma^k(A,x)$ this multiplier or $\sigma^k(u,x)$ for $A = A_u$.

For a~representation $\pi$ of $\g$, we denote by $L^2(\pi)$ the underlying Hilbert space and $\CCC^{\infty}(\pi)$ the subspace of smooth vectors.
If $\pi\in \widehat{\g}$ is non-trivial, we define $\sigma^k(A,x,\pi)$ by
\[\sigma^k(A,x,\pi)(\pi(f)\xi) = \pi\bigl(\sigma^k(A,x)f\bigr)\xi, f\in \mathscr{S}_0\bigl(\g,\Omega^1\g\bigr), \xi \in L^2(\pi).\]
This is well defined and gives a~linear map $\sigma^k(A,x,\pi) \colon \CCC^{\infty}(\pi) \to \CCC^{\infty}(\pi)$. This principal symbol is compatible with the convolution product of distributions. If $v\in \mathcal{E}^{'\ell}(M\times \g)$, then
\[\sigma^k(u,x,\pi)\sigma^{\ell}(v,x,\pi) = \sigma^{k+\ell}(u\ast v, x,\pi).\]
Similarly with the adjoint $\sigma^k(u,x,\pi) = \sigma^{\bar{k}}\bigl(u^*,x,\pi\bigr)^*$.

Before stating the next result, recall that $\mathcal{S}_0\bigl(\g,\Omega^1\g\bigr)$ is a~dense subalgebra of the $C^*$-algebra $C^*_0(\g)$ defined as the kernel of the trivial representation in $C^*(\g)$ (which itself is a~completion of $\CCC^{\infty}_c\bigl(\g,\Omega^1\g\bigr)$).

\begin{thm}[{\cite[Theorem 3.18]{AMYMaximalHypoellipticity}}] Let $u\in \mathcal{E}^{'k}(M\times \g)$ with $\Re(k)= 0$ then $\sigma(u,x)$ extends to a~multiplier of $C^*_0(\g)$. Moreover, for every non-trivial $\pi\in \hat{\g}$, $\sigma^0(u,x,\pi)$ extends to a~bounded operator on $L^2(\pi)$ and
\[\sup_{x\in M}\| \sigma(u,x)\|_{C^*_0(\g)} = \sup_{x\in M,\, \pi\in\widehat{\g}\setminus\{0\}} \|\sigma(u,x,\pi)\|_{\mathcal{B}(L^2(\pi))}<+\infty.\]
Finally, we have $\sigma^0(u) \colon x \mapsto \sigma(u,x) \in \CCC_c(M)\otimes M(C^*_0(\g))$.
\end{thm}

\begin{proof}
The multiplier aspect and the norm relations are the content of~\cite[Theorem~3.18]{AMYMaximalHypoellipticity} as well as the continuity aspect in the last part. The compact support comes from the fact that $u$ is compactly supported along $M$ thus so is $A_u$.
\end{proof}

We can then form the $C^*$-algebra $\Sigma(M\times \g)$ by taking the closure of the subalgebra formed by the $\sigma^0(u)$, $u\in \mathcal{E}^{'0}$ inside of $\CCC_0(M)\otimes M\bigl(C^*_0(\g)\bigr)$. This algebra is isomorphic to $\CCC_0(M) \otimes \Sigma(\g)$ where $\Sigma(\g)$ is defined analogously, using \smash{$\mathcal{E}^{'0}(\g)$} instead. From~\cite[Proposition 5.6]{FermanianFischerGradedLie} (see also~\cite[Theorem 5.8]{CrenFilteredCrossedProduct}), we get that the evaluations of symbols at points $(x,\pi)\in M \times \bigl(\widehat{\g}\setminus 0\bigr)$ define irreducible representations of $\Sigma(M\times \g)$ and all of them arise that way. Moreover, two such representations are equivalent to one another if and only if they are on the same orbit for the inhomogeneous dilations $(\delta_{\lambda})_{\lambda >0}$. Consequently, we get
\[\widehat{\Sigma(M\times \g)} \cong M \times {(\widehat{\g}\setminus 0)}/{\R^*_+}.\]

Going back to the filtered calculus, not all these representations make sense for the principal symbol of an operator in $\Psi^k_{\Fb}(M)$. Indeed, if $u\in \mathcal{E}^{'k}(\U)$ is the global lift of a~given operator, we want to only consider the elements $x\in M$, $\pi \in \widehat{\g}$ for which $\sigma(u,x,\pi)$ does not depend on the choice of $\g$, $\U$ or $u$. The representations that satisfy this condition are exactly the ones coming from the Helffer--Nourrigat cone. Let $\Sigma_{\Fb}(M)$ be the quotient of $\Sigma(M\times \g)$ corresponding to the closed subset \smash{${\HN_0(\Fb)}/{\R^*_+}$}. This algebra is independent of the choice of $\g$ by Proposition~\ref{Proposition : Independence HN} (see also the remark following that proposition).

We can see the algebra $\Sigma_{\F}(M)$ as a~$\CCC_0(M)-C^*$-algebra, whose fiber at a~point $x \in M$ is $\Sigma(\gr(\Fb)_x)$ if $x \in M_{\rm reg}$. If $x$ is a~singular point, then $\Sigma_{\F}(M)_{x}$ is a~quotient of $\Sigma(\gr(\Fb)_x)$ corresponding to the closed subset
\[
{\HN_0(\Fb)_x}/{\R^*_+} \subset  {\bigl(\widehat{\gr(\Fb)_x}\setminus \{0\}\bigr)}/{\R^*_+}.
\]
 We also have defined a~map $\sigma^0 \colon \mathcal{E}^{'0}(\mathcal{\mathbb{U}}) \to \Sigma_{\F}(M)$.
In order to prove Theorem~\ref{Theorem:SpectrumSymbols}, all that remains is to prove that this map factors through~$\Psi^0_{\F}(M)$ in the sense that if $P,Q \in \mathcal{E}^{'0}(\mathbb{U})$ satisfy $\ev_{1*}(P) - \ev_{1*}(Q)\in \CCC^{\infty}_c\bigl(M\times M,\Omega^{1/2}\bigr)$, then $\sigma^0(P)=\sigma^0(Q)$. Once this is done, one has to show that the symbol map induced on $\Psi^0_{\F}(M)$ extends continuously to the $C^*$-closure~$\Psi_{\F}(M)$ and that we obtain the exact sequence~\eqref{Equation:PseudosExactSequence}. Theorem~\ref{Theorem:SpectrumSymbols} is thus a~consequence of the following result.

\begin{thm}[Androulidakis, Mohsen, Yuncken~\cite{AMYMaximalHypoellipticity}] \label{Theorem: AMY key result}
Let $P \in \Psi^0_{\Fb}(M)$, $u \in \mathcal{E}^{'0}(M\times \g)$ be the restriction over $M\times \g \times 0$ of any global lift. The class of $\sigma^0(u)$ in $\Sigma_{\Fb}(M)$ does not depend on any choice of lift and graded basis. Moreover, the map $\sigma^0 \colon \Psi^0_{\Fb}(M) \to \Sigma_{\Fb}(M)$ described this way is a~continuous $*$-algebra homomorphism. If $\Psi_{\Fb}(M)$ denotes the completion of $\Psi^0_{\Fb}(M)$ within the bounded operators on $L^2(M)$, we get the exact sequence
\[\xymatrix{0 \ar[r] & \mathcal{K}\bigl(L^2(M)\bigr) \ar[r] & \Psi_{\Fb}(M) \ar[r]^{\sigma^0} & \Sigma_{\Fb}(M) \ar[r] & 0.}\]
\end{thm}

\begin{proof}
 The proof can be found in~\cite{AMYMaximalHypoellipticity}. The fact that the principal symbol is well defined and its continuity are given by Theorem 3.34. The exact sequence is Theorem 3.39.
\end{proof}

\subsection[Decomposition series for the algebra of principal symbols and pseudodifferential operators]{Decomposition series for the algebra of principal symbols\\ and pseudodifferential operators}

Assume once again that $(M,\F)$ is a~finitely generated singular Lie filtration. We fix a~global graded basis $\beta\colon\g\to \mathfrak{X}(M)$ which by the Rothschild--Stein construction, comes with a~Jordan--H\"older basis (for $\g$) compatible with the grading. Recall that we have an injection
\[\beta^* \colon\ \HN_0(\F) \hookrightarrow M \times\bigl(\widehat{\g}\setminus 0\bigr).\]
Let $\widetilde{U}_1\subset \cdots\subset \widetilde{U}_D = \widehat{\g}\setminus\{0\}$ be the increasing unions in the fine stratification (we have removed~$0$ on each fibers, it corresponds to the trivial representation and only affects the last strata). Let~${\widetilde{\Omega}_j = \widetilde{U}_j \setminus \widetilde{U}_{j-1}}$ be the corresponding strata.

Let $\Sigma_j \triangleleft\Sigma_{\Fb}(M)$ be the ideal such that
\smash{$\widehat{\Sigma_j} = \bigl(\beta^*\bigr)^{-1}\bigl( M\times {\widetilde{U}_j}/{\R^*_+}\bigr)$}.
This defines an ideal of~$\Sigma_{\Fb}(M)$ because it corresponds to an open subset of the spectrum.

\begin{lem}
 There is an isomorphism
 \[{\Sigma_j}/{\Sigma_{j-1}} \cong \CCC_0\bigl(\bigl(\beta^*\bigr)^{-1}\bigl( M\times {\widetilde{\Omega}_j}/{\R^*_+}\bigr), \mathcal{K}_j\bigr),\]
 where $\mathcal{K}_j$ is the algebra of compact operators on a~separable Hilbert space, of infinite dimension if $j < D$ and dimension $1$ if $j = D$.
\end{lem}
\begin{proof}
 Define $\widetilde{\Sigma}_j \triangleleft \Sigma(M \times \g)$ whose spectrum corresponds to the open set \smash{$M\times {\widetilde{U}_j}/{\R^*_+}$}. Then by construction the map $\Sigma(M \times \g) \to \Sigma_{\Fb}(M)$ sends $\widetilde{\Sigma}_j$ to $\Sigma_j$. We therefore have a~quotient map
${\widetilde{\Sigma}_j}/{\widetilde{\Sigma}_{j-1}}\twoheadrightarrow {\Sigma_j}/{\Sigma_{j-1}}$.

 Now from~\cite[Theorem 5.16]{CrenFilteredCrossedProduct}, we know that
 \[{\widetilde{\Sigma}_j}/{\widetilde{\Sigma}_{j-1}} \cong \CCC_0\bigl( M\times {\widetilde{\Omega}_j}/{\R^*_+}, \mathcal{K}_j\bigr).\]
 The construction of $\Sigma_{\Fb}(M)$ then gives that this quotient has the desired form.
\end{proof}

Now, recall from Section~\ref{Section : Stratification HN Cone} that we stratified $\HN(\Fb)$ by taking the sets $\bigl(\beta^*\bigr)^{-1}\bigl( M\times \widetilde{\Omega}_j\bigr)$,
and re-indexing them to only consider those that are non-empty. With sets described here, we get a~stratification of $\HN_0(\Fb)$ (again, it only affects the last non-empty strata).
Let $\Omega_1, \dots, \Omega_d$ be the non-empty Pedersen strata in $\HN_0(\Fb)$ and $U_j = \bigcup_{k\leq j}\Omega_k$. Let $\Sigma_j \subset \Sigma_{\F}(M)$ be the ideal corresponding to the open subset \smash{${U_j}/{\R^*_+}$}. This is nothing but one of the previous $\Sigma_k$ after re-indexing to avoid repeating the same ideal several times. With these new notations, we have
\[{\Sigma_j}/{\Sigma_{j-1}} \cong \CCC_0 ({\Omega_j}/{\R^*_+}, \mathcal{K}_j ).\]
We summarize this discussion in the following theorem, which is similar to Theorem~\ref{Theorem:SolvabilityCTF}.

\begin{thm}
 Let $(M,\F)$ be a~finitely generated singular Lie filtration. We fix generators and let $\g$ be the corresponding global Lie basis. Let $\varnothing = U_0 \subset U_1 \subset \cdots \subset U_d = \HN_0(\F)$,
 be the Pedersen stratification and denote by $\Omega_j = U_j\setminus U_{j-1}, 1\leq j \leq d$ the corresponding strata. Let~${\Sigma_j\triangleleft \Sigma_{\F}(M)}$ be the ideal corresponding to the open set \smash{${U_j}/{\R^*_+}$}. Then we have
 \[{\Sigma_j}/{\Sigma_{j-1}} \cong \CCC_0 ({\Omega_j}/{\R^*_+},\mathcal{K}_j).\]
 Here $\mathcal{K}_j$ is the algebra of compact operators on a~separable Hilbert space, of infinite dimension for $j < d$, and dimension $1$ for $j = d$.
\end{thm}

\begin{cor}
 The $C^*$-algebras $\Sigma_{\Fb}(M)$ and $\Psi_{\Fb}(M)$ are $($explicitly$)$ solvable.
\end{cor}

\begin{proof}
 We have proved it for $\Sigma_{\Fb}(M)$, the result then follows for $\Psi_{\Fb}(M)$ by the exact sequence of Theorem~\ref{Theorem: AMY key result}.
\end{proof}

\begin{rem}
 This result also works with other types of decompositions. In the case where the coarse stratification of Puk\'anszky gives trivial subquotients like the fine one (e.g., if $\g$ is a~step-2 group and has flat orbits like Examples~\ref{Ex: Heisenberg group} and~\ref{Ex: Complex Heisenberg group}).
\end{rem}

From the decompositions of $C^*(\T\F)$ and $\Sigma_{\F}(M)$, one can extend a~result of Ewert~\cite{Ewert} whose statement is the same as the following corollary but for equiregular filtrations.

\begin{cor}
 Let $C^*_0(\T\F)$ be the ideal of $C^*(\T\F)$ corresponding to the open set $\HN_0(\F)\subset \HN(\F)$. Then the action $\delta\colon \R^*_+\act C^*_0(\T\F)$ is free and proper, we get
 \[
 \Sigma_{\F}(M)\otimes \mathcal{K} \cong C^*_0(\T\F)\rtimes_{\delta}\R^*_+,
 \]
 where $\mathcal{K}$ is the algebra of compact operators on a~separable infinite-dimensional Hilbert space.
\end{cor}

\begin{proof}
 Denote by $0 =\Sigma_0 \triangleleft\Sigma_1\triangleleft\cdots\triangleleft \Sigma_d = \Sigma_{\F}(M)$
 and $0 =J_0 \triangleleft J_1\triangleleft\cdots\triangleleft J_d = C^*_0(\T\F)$,
 the decompositions into nested ideals obtained from a~Pedersen stratification of $\HN_0(\F)$. By induction, it suffices to prove that $J_k \rtimes_{\delta}\R^*_+ \cong \Sigma_k\otimes \mathcal{K}$. Note first that the ideals $J_k$ are indeed preserved by the $\R^*_+$-action since they correspond to parts of the spectrum that are stable by the corresponding action.

 If we have an isomorphism for some $k$, then we can deduce that we still have one for ${k+1}$ by looking at \smash{${J_k}/{J_{k-1}}$}. Since $\Omega_k$ is homeomorphic to \smash{${\Omega_k}/{\R^*_+}\times \R^*_+$} (we can choose for instance a~homogeneous quasi-norm on $\gr(\Fb)^*$ so that this corresponds to polar coordinates), the isomorphism
 \[\CCC_0(\Omega_k,\mathcal{K}_k)\rtimes_{\delta}\R^*_+ \cong \CCC_0({\Omega_k}/{\R^*_+},\mathcal{K}_k)\otimes \mathcal{K}\]
 is a~direct consequence of Takai duality~\cite{Takai}.
\end{proof}

\subsection*{Acknowledgements}

The author would like to thank the anonymous referees for their comments. They helped improving the clarity and readability of this article. The author is also grateful to Iakovos Androulidakis, Omar Mohsen and Bob Yuncken for their explanations on their paper.

\pdfbookmark[1]{References}{ref}
\LastPageEnding

\end{document}